\documentclass[11pt]{amsart}

\usepackage{graphicx} 
\usepackage[english]{babel}     
\usepackage[utf8]{inputenc}
\usepackage{amsmath}
\usepackage{amsfonts}
\usepackage{amssymb}
\usepackage{hyperref}
\usepackage{xcolor}
\usepackage{comment}
\usepackage{enumerate}
\usepackage{tikz-cd}
\usepackage{pgfplots}
\usepackage{doi}
\usepackage[normalem]{ulem}

\usepackage{todonotes}
\usepackage{comment}
\specialcomment{auxproof}
{\mbox{}\newline\textbf{BEGIN: AUX-PROOF}\dotfill\newline}
{\mbox{}\newline\textbf{END: AUX-PROOF}\dotfill\newline}

\usepackage{amsthm}
\newtheorem{theorem}{Theorem}[section]

\newtheorem{lemma}[theorem]{Lemma}
\newtheorem{proposition}[theorem]{Proposition}
\newtheorem{corollary}[theorem]{Corollary}
\newtheorem{question}[theorem]{Question}

\theoremstyle{definition}
\newtheorem{definition}[theorem]{Definition}

\newtheorem{remark}[theorem]{Remark}

\newcommand{\Tr}{\textrm{Tr}}

\newcommand{\supp}{{\rm supp}}

\newcommand{\Besov}{\dot{B}}

\title[Higher order perturbation estimates]{Higher order perturbation estimates in quasi-Banach Schatten spaces through wavelets}

\author{Martijn Caspers}
\author{Emiel Huisman}

\address{TU Delft, EWI/DIAM,
	P.O.Box 5031,
	2600 GA Delft,
	The Netherlands}

\email{M.P.T.Caspers@tudelft.nl}

\email{E.Huisman-1@student.tudelft.nl}

\date{\today. MSC2010 keywords: 47B10, 47L20, 47H60. MC is supported by the NWO Vidi grant
VI.Vidi.192.018 ‘Non-commutative harmonic analysis and rigidity of operator algebras’.}

\begin{document}

\maketitle

\begin{abstract}
 Let $n \in \mathbb{N}_{\geq 1}$. Let $1 \leq p_1, \ldots, p_n < \infty$ and set the H\"older combination $p := (p_1; \ldots ; p_n) := \left( \sum_{j=1}^n  p_j^{-1} \right)^{-1}$.
 Assume further that  $0 < p \leq 1$ and that for  the H\"older combinations of $p_2$ to $p_n$ and $p_1$ to $p_{n-1}$ we have, 
 \[
 1 \leq (p_2; \ldots ; p_n), (p_1; \ldots ; p_{n-1}) < \infty. 
 \]
 Then there exists a constant $C> 0$ such that for every   $f \in C^n(\mathbb{R}) \cap \Besov_{\frac{p}{1-p}, p}^{n-1 + \frac{1}{p}}$ with  $\Vert f^{(n)} \Vert_\infty < \infty$ we have 
\[
\Vert  T_{f^{[n]}}: S_{p_1} \times \ldots \times S_{p_n} \rightarrow S_p    \Vert  \leq C ( \Vert f^{(n)} \Vert_\infty +  \Vert f \Vert_{\Besov_{\frac{p}{1-p}, p}^{n-1 + \frac{1}{p}}}).
\]
 Here $S_q$ is the Schatten von Neumann class, $\Besov_{p,q}^s$ the  homogeneous Besov space, and $T_{f^{[n]}}$ is the multilinear Schur multiplier of the $n$-th order divided difference function. In particular, our result holds for $p=1$ and any $1 \leq p_1, \ldots, p_n < \infty$ with $p = (p_1; \ldots; p_n)$. 
\end{abstract}

\section{Introduction}

In the 1950's and early 1960's  M.G. Krein \cite{Krein1,Krein2}  and I. Lifschitz \cite{Lifschitz} laid the foundation of perturbation theory based on the, at the time, relatively new theory of operator algebras (see also Daleckii and S.G. Krein \cite{DaletskiKrein1, DaletskiKrein2}). Their ideas were crucial in the analysis of spectral properties of a Hamiltonian that is perturbed with a  potential function. This led in particular to the introduction of the spectral shift function. Perturbation theory is now  integral to quantum mechanics and  has a wide range of applications in mathematics, mathematical physics and close connections to modern noncommutative geometry and spectral action (see e.g. \cite{Suijlekom, Nuland2}). Intimately related, is also the analysis of commutator estimates,   differentiability of continuous functional calculus and noncommutative Taylor expansions (a detailed account is given   below).  For an overview of the field until the turn of the millennium and survey of spectral shift we refer to \cite{Gesztesy, BirmanP}.  

Of central importance in the theory is the concept of a divided difference function 
\begin{equation}\label{Eqn=IntroFn}
\left\{ f^{[1]}(s,t) := \frac{f(s) - f(t)}{s - t} \right\}_{s,t \in \mathbb{R}, s \not = t}
\end{equation}
and an associated linear map   which is either a {\it Schur multiplier} or a closely related {\it double operator integral} \cite{SkripkaTomskova}. The boundedness of such Schur multipliers, acting on a noncommutative $L^p$-space,  implies directly three of the problems going back to Krein and Lifschitz: (1) commutator estimates, (2) Lipschitz properties of functional calculus and (3) the existence of spectral shift. Though this was known ever since the founding work of Lifschitz and Krein the theory has seen a development of more than 70 years with a number of spectacular highlights going well into the last decade. Especially the close connection with harmonic analysis and transference techniques has solved many of the most important problems in this direction (see e.g. \cite{PS-Acta}, \cite{PSS-Inventiones}, \cite{CPSZ-AJM}, \cite{CGPT-Annals} for some prominent examples).

Let us concretely state the problem that we shall be addressing and an overview of the results so far. We consider the symbol \eqref{Eqn=IntroFn} with $f$ Lipschitz and Lipschitz constant $\Vert f' \Vert_\infty < \infty$. Our main question, in the linear case,  is under which further regularity conditions of $f$ we have that 
\begin{equation}\label{Eqn=Tf}
T_{f^{[1]}}: S_p \rightarrow S_p: \{ x_{s,t} \}_{s,t \in \mathbb{R}} \mapsto \{ f^{[1]}(s,t) x_{s,t} \}_{s,t \in \mathbb{R}},   
\end{equation}
is bounded on the Schatten von Neumann class $S_p$. Here we write  $\{ x_{s,t} \}_{s,t \in \mathbb{R}}$ for the kernel, if existent,  of an operator on $L^2(\mathbb{R})$, details can be found in Section \ref{Sect=Preliminaries}.  Now if $p = 1,   \infty$ then \eqref{Eqn=Tf} was conjectured to be bounded in \cite{KreinConjecture}. The conjecture was then disproved by Farforovskaya (see \cite{Farforovskaya1, Farforovskaya2, Farforovskaya3}). In fact already for $f$ the absolute value map the linear operator   \eqref{Eqn=Tf}  was shown to be unbounded by Kato \cite{Kato} and Davies \cite{Davies}.

The first positive results are in fact also the most relevant ones to this   paper. In \cite{BirmanSolomyak} Birman and Solomyak  proved that for $p=1$ we have that $\eqref{Eqn=Tf}$ is bounded for any $f \in C^{1+\varepsilon}, \varepsilon >2$ and so in particular if $f$ is a $C^2$ function that is Lipschitz boundedness holds. This result was then improved by Peller \cite{Peller} who showed boundedness of \eqref{Eqn=Tf} at $p=1$ whenever $f$ belong to the Besov class $\Besov^1_{\infty 1}$. Peller's result is our result in the linear case; but we emphasize that we do not give a new proof but in fact we use Peller his result as a starting point for an inductive reduction procedure to the higher order case. Peller's result was also revisited recently in \cite{McDonaldSukochev}. Slightly weaker sufficient conditions were also given in \cite{Arazy}.

Thereafter for $1<p <\infty$ a complete solution was found in \cite{PS-Acta} and boundedness of \eqref{Eqn=Tf} holds if and only if $f$ is Lipschitz.  That means that boundedness follows under minimal regularity conditions on $f$. After \cite{PS-Acta} the endpoint estimates in weak $L^1$- and BMO-spaces as well as best constants were found in \cite{CMPS-JFA,  CSZ-AIF,  CJSZ-JFA, CPSZ-AJM}. A remarkably short proof of the main result of \cite{PS-Acta} was found recently in  \cite{CGPT-Annals, GPPR} where this follows from a general H\"ormander-Mikhlin-Schur multiplier theorem   (see also \cite{ParcetICM}).  The multilinear analogue of this problem was solved in \cite{PSS-Inventiones}.    Important applications to higher order Fr\'echet differentiability of Schatten norms were consequently found in \cite{PotapovSukochev-Frechet, PSTZ-Transactions}.
The result from \cite{PSS-Inventiones} was sharpened in the bilinear case in \cite{CaspersReimann}. For the absolute value estimate a weak $L^1$-end point estimate was then obtained in \cite{CSZ-Israel}. Finally we mention that in the multilinear case at $p=\infty$  characterisations of Schur multipliers and operator integrals are available that go back to Grothendieck, see e.g. \cite{JuschenkoTodorovTurowska, Coine}.

For a long time it was not known whether in  the regime $0 < p < 1$ bounds of  Schur multipliers of divided differences can be found. In this range the Schatten spaces become quasi-Banach spaces and fail to be locally convex. In \cite{McDonaldSukochev} a fully satisfying answer was provided stating that indeed \eqref{Eqn=Tf} is bounded in the range $0 < p < 1$ as long as $f$ is a Lipschitz function belonging to the homogeneous Besov space $\Besov_{p^\sharp, p}^{\frac{1}{p}}$ with $p^\sharp = \frac{p}{1-p}$.   Their work was preceded by important work  on  $L^p$-spaces with $p < 1$, in particular by  Aleksandrov and Peller \cite{AleksandrovPeller1, AleksandrovPeller2, Peller87}, and further  \cite{Rotfeld1, Rotfeld2, Huang,  Ricard, PisierRicard}.

 Here we provide a first analysis in the multilinear situation. In Euclidean harmonic analysis it is now a well established fact that many natural multilinear Fourier multipliers admit bounds in $L^p$-spaces with $p$ in the quasi-Banach regime. This holds true for H\"ormander-Mikhlin type multipliers \cite{GrafakosTorres} or bilinear Hilbert transforms \cite{Lacey}. In the noncommutative or completely bounded setting such properties simply fail  as was shown in  \cite[Section 5]{CKV-Canada} and thus   one is required to put further regularity conditions on the function $f$ just as in Peller's original result \cite{Peller}.  

\vspace{0.3cm}

The main theorem we prove is the following (see Theorem \ref{Thm=Main}). 

\begin{theorem}\label{Thm=MainIntro}
 Let $n \in \mathbb{N}_{\geq 1}$, let $1 \leq p_1, \ldots, p_n < \infty$ and set $p := (p_1; \ldots ; p_n)$.
 Assume further that, 
 \begin{equation}\label{Eqn=RangeOfPs}
0< p \leq 1, \:\: 1 \leq (p_2; \ldots ; p_n), (p_1; \ldots ; p_{n-1}) < \infty.
 \end{equation}
There exists a constant $C>0$ such that for every $f \in C^n(\mathbb{R}) \cap \Besov_{\frac{p}{1-p}, p}^{n-1 + \frac{1}{p}}$ such that $\Vert f^{(n)} \Vert_\infty < \infty$ we have that,  
\[
\Vert  f^{[n]}   \Vert_{\mathfrak{m}_{p_1, \ldots,  p_n}} \leq C(  \Vert f^{(n)} \Vert_\infty +  \Vert f \Vert_{\Besov_{\frac{p}{1-p}, p}^{n-1 + \frac{1}{p}}}).
\]
\end{theorem}

This theorem is new already in case $p = 1$ and we are still in the Banach space setting. In this case \eqref{Eqn=RangeOfPs} is void and  our theorem gives sufficient regularity conditions to yield a bound on the norms of  Schur multipliers of divided differences. In case $p \in (0,1)$ our result is to the knowledge of the authors the  first noncommutative multilinear result of a Schur multiplier whose recipient space is a quasi-Banach $L^p$-space. However the question whether the restriction \eqref{Eqn=RangeOfPs} can be removed at the expense of having stronger Besov regularity assumptions on $f$  remains an open problem, see Section \ref{Sect=Discussion}.

The techniques we use in this paper are largely inspired by the wavelet approach that is taken in \cite{McDonaldSukochev} (see also \cite{McDonaldWavelets}). The main novelty is a reduction theorem (Theorem \ref{Thm=InductiveStep}) that shows that the symbols appearing in   Theorem \ref{Thm=MainIntro} for $n$ exponents split into two parts: (1) symbols that are $n-1$-th order divided differences and (2) a genuinely $n$-linear multiplier that is concentrated on a block diagonal. The first part can be estimated inductively and for the second part we carry out a wavelet analysis. It turns out that both parts can be controlled by the same homogeneous Besov norm, with the same exponents. This eventually results in Theorem \ref{Thm=MainIntro}.

\vspace{0.3cm}

\noindent {\bf Contents.} In Section \ref{Sect=Preliminaries} we recall the preliminaries. Section \ref{Sect=CoreEstimate} contains the core estimate behind \eqref{Eqn=RangeOfPs} and most of the new results of this paper. Then in Section \ref{Sect=MainResult} we prove Theorem \ref{Thm=MainIntro} using a multilinear analogue of the strategy in \cite{McDonaldSukochev}.

\vspace{0.3cm}

\noindent {\bf Acknowledgments.} The authors wish to thank the anonymous referee for a detailed report that led to an improvement of our manuscript.

\section{Preliminaries}\label{Sect=Preliminaries}

\subsection{General notation} To ensure compatibility with the existing literature we assume all Hilbert spaces are complex and separable. 

We let $\chi_A$ be the indicator function on a set $A$. $\mathbb{T}$ denotes the torus which we identify with the unit circle on $\mathbb{C}$. We let $C^n(\mathbb{R})$ or  $C^n(\mathbb{T})$ be the space of $n$ times continuously differentiable real valued functions.  We let $C_c^n(\mathbb{R})$ be those functions that have moreover compact support. For a function $\psi \in L^2(\mathbb{R})$ or  $\psi \in L^2(\mathbb{T})$ we denote $\widehat{\psi}$ for its Fourier transform which lies in $L^2(\mathbb{R})$ and $\ell^2(\mathbb{Z})$, respectively.  

The symbol $\preceq$ stands for an inequality that holds up to a constant where the constant may differ line by line. The  constants may depend on preset choices of objects or quantities that appear in the statement of a theorem like $\phi, p, R, n$ but those dependencies do not affect the proof; in Section \ref{Sect=CoreEstimate}  for instance it is only relevant that the inequalities that appear are  independent of $\alpha$ and $\lambda$. We sometimes write expressions like  $\preceq_{\phi,n}$ to clarify explicitly that a constant depends on $\phi$ and $n$.  We use $\approx$ in case we have equality up to a constant. 

\subsection{Homogeneous Besov spaces} Denote by $\mathcal{S}(\mathbb{R})$ the algebra of all Schwartz class functions on $\mathbb{R}$ with its usual Fr\'echet topology and dual  $\mathcal{S}'(\mathbb{R})$ which is the space of tempered distributions. Let $\phi: \mathbb{R} \rightarrow \mathbb{R}$ be  smooth, supported in $[-2, -1+ \frac{1}{7}) \cup (1 - \frac{1}{7}, 2]$ and identically equal to 1 in the set $[-2 + \frac{2}{7}, -1) \cup (1, 2 - \frac{2}{7}]$. 
Assume further that 
\[
\sum_{j \in \mathbb{Z}}  \phi(2^{-j} \xi) = 1, \qquad \xi \not = 0. 
\]
Let $\Delta_j, j \in \mathbb{Z}$ be the operator on $\mathcal{S}'(\mathbb{R})$ of Fourier multiplication by the function $\xi \mapsto \phi(2^{-n} \xi)$; i.e. $\Delta_j f$ with $f \in \mathcal{S}'(\mathbb{R})$ is the tempered distribution whose (distributional) Fourier transform equals $\phi(2^{-n} \: \cdot \: ) \widehat{f}$ where $\widehat{f}$ is the Fourier transform of $f$. The series $\{ \Delta_j f \}_{j \in \mathbb{Z}}$ is called the Littlewood-Payley decomposition of $f$.
Now let $s \in \mathbb{R}$ and $p,q \in (0, \infty]$. We consider the homogeneous Besov space $\Besov^s_{p,q} = \Besov^s_{p,q}(\mathbb{R})$ of distributions $f \in \mathcal{S}'(\mathbb{R})$ for which 
\begin{equation}\label{Eqn=BesovDfn}
\Vert f \Vert_{\Besov^s_{p,q}} := \Vert  \{ 2^{js} \Vert \Delta_j f \Vert_p \}_{j \in \mathbb{Z}} \Vert_{\ell^q(\mathbb{Z})} < \infty.
\end{equation}
A concrete characterisation of Besov spaces, for locally integrable functions, in terms of wavelets shall be recalled in Section \ref{Sect=MainResult}. 


\subsection{Divided differences} 

\begin{definition}
For $f \in C^{n}(\mathbb{R})$ we inductively define the divided difference functions for $k = 1, \ldots, n$ as 
\begin{equation}\label{Eqn=DivDiffDef}
f^{[k]}(t_0, \ldots, t_k) = 
\frac{f^{[k-1]}(t_0, t_2, \ldots, t_k) - f^{[k-1]}(t_1, t_2 \ldots, t_k)}{t_0 - t_1}, \qquad t_0, \ldots,  t_k \in \mathbb{R}, t_0 \not = t_1.
\end{equation}
For $t:= t_0 = t_1$ we set $f^{[k]}(t, t, t_2, \ldots, t_n) = \frac{d}{dt} f^{[k-1]}(t, t_2, \ldots, t_n)$. For $f \in C^n(\mathbb{T})$ and $t_i \in \mathbb{T}$ the same definition defines $f^{[k]}$ on the torus; see the proof of Proposition \ref{Prop=Torus} for explicit expressions.
\end{definition}

Without further notification we shall further use that divided differences are permutation invariant (\cite{SkripkaTomskova, DevLor}), meaning that for any permutation $\sigma$ of the integer numbers between $0$ and $n$ we get that 
\[
f^{[n]}(t_0, \ldots, t_n) = f^{[n]}(t_{\sigma(0)}, \ldots, t_{\sigma(n)}). 
\]
In particular, for $i \not = j$, 
\[
f^{[k]}(t_0, \ldots, t_k) = 
\frac{f^{[k-1]}(t_0, \ldots, \widehat{t_j}, \ldots,  t_k) - f^{[k-1]}(t_1,  \ldots, \widehat{t_i}, \ldots,  t_k)}{t_i - t_j}, 
\]
where the hat notation, i.e. $\widehat{t_i}$ and $\widehat{t_i}$, indicate that the variables $t_i$ and $t_j$ are omitted.

\subsection{Multilinear maps} Let $X_1, \ldots, X_n, X$ be (quasi-)Banach spaces. The bound of a multilinear map $T: X_1 \times \ldots \times  X_n \rightarrow X$ is given by 
\[
\Vert T \Vert = \sup_{x_i \in X_i, \Vert x_i \Vert_{X_i} = 1}  \Vert T(x_1, \ldots, x_n) \Vert_X. 
\]

\subsection{Schatten classes} In this paper we consider noncommutative $L^p$-spaces associated with the bounded operators $B(L^2(\mathbb{R}))$ on the Hilbert space $L^2(\mathbb{R})$. For $x \in B(L^2(\mathbb{R}))$   and $0 < p < \infty$ we set 
\[
\Vert x \Vert_p = {\rm \Tr}(  \vert x \vert^p)^{\frac{1}{p}}.
\]
Then set 
\[
S_p = \{ x \in B(L^2(\mathbb{R})) \mid \Vert x \Vert_p < \infty \}. 
\]
Then $S_p$ consists of compact operators whose singular values form a sequence in $\ell^p$. 
In case $1\leq p <\infty$ these spaces are Banach spaces and in case $0<p < 1$ these spaces are quasi-Banach spaces satisfying the qausi-Banach inequality
\begin{equation}\label{Eqn=Quasi-Banach}
\Vert x + y \Vert_p^p \leq \Vert x \Vert_p^p + \Vert y \Vert_p^p, \qquad x,y \in S_p.
\end{equation}
Note that at the threshold case $p=1$ this is the usual triangle inequality. 
Similarly $\Vert \sum_{j =1}^\infty x_j \Vert_p^p \leq \sum_{j =1}^\infty \Vert  x_j \Vert_p^p$ for infinite converging series of $x_j \in S_p, 0<p<1$. 
For constants $0 < p_1, \ldots, p_n < \infty$ we write $(p_1; \ldots; p_n) = ( \sum_{j = 1}^n  p_j^{-1} )^{-1}$ for their H\"older combination. For $x_1 \in S_{p_1}, \ldots, x_n \in  S_{p_n}, p_i \in (0, \infty)$ we have $x_1 \ldots x_n \in S_p$ with  $p = (p_1; \ldots; p_n)$ and the H\"older estimate holds, 
\[
\Vert x_1 \ldots x_n \Vert_p \leq \prod_{j=1}^n \Vert x_j \Vert_{p_j}.
\]
The space $S_2$ is a Hilbert space that can linearly be identified with $L^2(\mathbb{R}^2)$ by letting $\{ x_{s,t} \}_{s,t \in \mathbb{R}}$ in $L^2(\mathbb{R}^2)$ correspond to $x \in S_2$ determined by  
\[
\langle x \xi, \eta\rangle = \int_{\mathbb{R}^2}  x_{s,t} \xi(s) \overline{\eta(t)} ds dt, \qquad \xi, \eta \in L^2(\mathbb{R}). 
\]
We call $\{ x_{s,t} \}_{s,t \in \mathbb{R}}$ the kernel of $x$.

\subsection{Schur multipliers}  
We recall the following from \cite{CLS-AIF, CKV-Canada}, for which we recall that $S_p \subseteq S_q$ when $0 < p \leq q < \infty$ and this inclusion is dense; in fact the finite rank operators are contained in every $S_p$-space $p \in (0, \infty)$ as a dense subset. 

\begin{definition}
Let $\psi \in L^\infty(\mathbb{R}^{n+1})$ whose variables we label with index 0 to $n$. Consider the multilinear map 
\begin{equation}\label{Eqn=DfnSchur}
T_\psi: S_{2} \times \ldots \times S_{2} \rightarrow S_2
\end{equation}
that maps $x_1, \ldots, x_n \in S_2$ with kernels $\{ x_{1, s_0, s_1 } \}_{s_0, s_1 \in \mathbb{R}}, \ldots, \{ x_{n, s_{n-1}, s_n }\}_{s_{n-1}, s_n \in \mathbb{R}}$ to the operator $T_\psi(x_1, \ldots, x_n) \in S_2$ with kernel 
\[
\{  \int_{\mathbb{R}^{n-1}} \psi(s_0, \ldots, s_n)  x_{1, s_0, s_1} \ldots  x_{n, s_{n-1}, s_n} ds_1 \ldots ds_{n-1} \}_{s_0, s_{n} \in \mathbb{R}}.  
\]
The assignment \eqref{Eqn=DfnSchur} is bounded with norm $\Vert \psi \Vert_{\infty}$ (\cite{CLS-AIF}). We shall denote 
\[
\Vert \psi \Vert_{\mathfrak{m}_{p_1, \ldots, p_n}} := \Vert T_\psi:  S_{p_1} \times \ldots \times S_{p_n} \rightarrow S_p  \Vert, 
\]
which is finite if $T_\psi$ extends to a bounded multilinear map from $(S_{p_1} \cap S_2) \times \ldots \times (S_{p_n} \cap S_2) \rightarrow S_{p}$ to $S_{p_1} \times \ldots \times S_{p_n} \rightarrow S_p$;  otherwise $\Vert \psi \Vert_{\mathfrak{m}_{p_1, \ldots, p_n}} = \infty$. $\psi$ is called the symbol of the Schur multiplier $T_\psi$. 
\end{definition}

 \subsection{Comparison with multiple operator integrals}\label{Sect=Comparison} In the proof of Proposition \ref{Prop=WaveletReal} we need to appeal to the transformation formula from \cite[Theorem 2.7]{PSS-Advances}. As \cite{PSS-Advances} is stated in the framework of multiple operator integals we show that our Schur multipliers appear as special cases of these multiple operator integrals. 

Let $H$ be the unbounded self-adjoint  operator (see \cite[Example 10.36]{NeervenBook} for self-adjointness) acting on $L^2(\mathbb{R})$ with domain 
\[
D(H) = \{ \xi \in L^2(\mathbb{R}) \mid  \int_{\mathbb{R}} \vert t \xi(t) \vert ^2 dt < \infty \}, 
\]
given by $(H\xi)(t) = t \xi(t)$. We have $\sigma(H) = \mathbb{R}$ and for $a: \mathbb{R} \rightarrow \mathbb{C}$ a bounded Borel function the spectral calculus gives $a(H) \xi  = a \xi$ where $\xi \in L^2(\mathbb{R})$. Let $\mathfrak{U}_n^c(\mathbb{R})$ (resp. $\mathfrak{U}_n^c(\mathbb{T})$ ) denote the set of functions $\phi: \mathbb{R}^{n+1} \rightarrow \mathbb{C}$ (resp. $\phi: \mathbb{T}^{n+1} \rightarrow \mathbb{C}$) that admit a representation 
\begin{equation}\label{Eqn=Representation}
\phi(s_0, \ldots, s_n) = \int_\Omega \prod_{j=0}^n a_j(s_j, \omega) d\mu(\omega), 
\end{equation}
for some finite measure space $(\Omega, \mu)$ and bounded continuous functions $a_j( \cdot , \omega): \mathbb{R} \rightarrow \mathbb{C}$ (resp. $a_j( \cdot , \omega): \mathbb{T} \rightarrow \mathbb{C}$). For $\phi \in \mathfrak{U}_n^c(\mathbb{R})$ with representation \eqref{Eqn=Representation} we set the multiple operator integral 
\begin{equation}\label{Eqn=MOINew}
T_\phi^H: S_{p_1} \times \ldots \times S_{p_n} \rightarrow S_p, \qquad p = (p_1; \ldots; p_n), \quad  1 \leq p, p_1, \ldots, p_n \leq \infty,  
\end{equation}
as follows
\[
T_\phi^H(x_1, \ldots, x_n) = \int_\Omega a_0(H, \omega) V_1 a_1(H, \omega) \ldots V_n a_n(H, \omega) d\mu(\omega), \qquad x_i \in S_{p_i}, \quad i = 1, \ldots, n. 
\]
Because of our choice of $H$ we have that $a_j(H, \omega)$ is the multiplication operator with function $ a_j( \: \cdot \: , \omega)$. 

\begin{lemma}\label{Lem=Comparison}
Let $\phi \in \mathfrak{U}_n^c(\mathbb{R})$. Then for every $x_1, \ldots, x_n \in S_2$ we have 
\begin{equation}\label{Eqn=Comparison}
T_\phi(x_1, \ldots, x_n) = T_\phi^H(x_1, \ldots, x_n). 
\end{equation}
\end{lemma}
\begin{proof}
First note that as $S_2 \subseteq S_{2n}$ the right hand side of \eqref{Eqn=Comparison} can be interpreted as \eqref{Eqn=MOINew} with $p_1 = \ldots = p_n = 2n$ and $p =2$. Moreover, as the inclusion $S_2 \subseteq S_{2n}$ is contractive $T_\phi^H$ yields a bounded map $S_2 \times \ldots \times S_2 \rightarrow S_2$ and so does $T_\phi$ by construction. 

Now take $x_i \in S_2$ and choose a kernel  $x_i = \{ x_{i, s, t}\}_{s,t \in \mathbb{R}}$. Then, by the discussion preceding this lemma, multiplying matrix kernels and using Fubini we find, 
\[
\begin{split}
T_\phi^H(x_1, \ldots, x_n)  = & \int_\Omega a_0(H, \omega) x_1 a_1(H, \omega) \ldots x_n a_n(H, \omega) d\mu(\omega) \\
= &  \int_\Omega \{   \int_{\mathbb{R}^{n-1}} a_0(s_0, \omega) x_{1, s_0, s_1} a_1(s_1, \omega)  \ldots x_{n, s_{n-1}, s_n} a_n(s_n, \omega)  ds_1\ldots ds_{n-1} \}_{s_0, s_n} d\mu(\omega) \\
= &  \{   \int_{\mathbb{R}^{n-1}}   \phi(s_0, \ldots, s_n)   x_{1, s_0, s_1}    \ldots x_{n, s_{n-1}, s_n}    ds_1\ldots ds_{n-1} \}_{s_0, s_n} \\
= & T_\phi. 
\end{split}
\]
\end{proof}

\subsection{The Potapov-Skripka-Sukochev theorem}
The following theorem is the core  of the main result of \cite{PSS-Inventiones}.

\begin{theorem}[Remark 5.4 from \cite{PSS-Inventiones}]\label{Thm=PSS}
Let $n \in \mathbb{N}_{\geq 1}$ and let $\psi \in C^n(\mathbb{R})$ with $\Vert \psi^{(n)} \Vert_\infty < \infty$. Let  $1 < p, p_1, \ldots, p_n < \infty$ be such that $p = (p_1; \ldots; p_n)$.  We have 
\begin{equation}\label{Eqn=PSS-Estimate}
\Vert \psi^{[n]} \Vert_{\mathfrak{m}_{p_1, \ldots, p_n}}   < \infty.
\end{equation}
\end{theorem}

\begin{remark}
Explicit upper and lower bounds for \eqref{Eqn=PSS-Estimate} have been obtained in \cite{CaspersReimann}.   However, we shall not require these explicit bounds to derive our main theorem. 
\end{remark}

\section{Wavelet estimates} \label{Sect=CoreEstimate}
In this section we collect all estimates of multilinear Schur multipliers of higher order divided differences of an individual wavelet with sufficient regularity.  


\subsection{Diagonal multipliers}  We start by collecting a number of elementary estimates. 
Let 
\[
\rho: \mathbb{R} \rightarrow [0,1], 
\]
be a smooth function supported on $[-2, 2]$ and which equals 1 in the interval $[-1, 1]$. Let 
\begin{equation}\label{Eqn=RhoR}
\rho_R(\xi) = \rho(R^{-1} \xi),
\end{equation}
which is then smooth, supported on $[-2R, 2R]$ and equals  1 on $[-R, R]$. The symbol considered in the following lemma is considered to be 0 when $s=t$.

\begin{lemma}[\cite{McDonaldSukochev}]\label{Lem=OneMinRho}
For $1 \leq p < \infty, R>0$ we have that $\Vert \{  
\frac{1-\rho_R(s-t)}{s-t}
\}_{s,t \in \mathbb{R}} \Vert_{\mathfrak{m}_{p}} < \infty$.  
\end{lemma}
\begin{proof}  
Let $G(t) = \frac{1 - \rho_R(t)}{t}, t \in \mathbb{R}$ for which we interpret $G(0) = 0$. Note that $G \in L^2(\mathbb{R})$ and thus has a Fourier transform $\widehat{G}$. By \cite[Lemma 4.2.3]{McDonaldSukochev} we have $\widehat{G} \in L^1(\mathbb{R})$.   Then, a well-known estimate that is recorded in  \cite[Proposition 4.2.2.(ii)]{McDonaldSukochev} yields, 
\[
\Vert \{  
\frac{1-\rho_R(s-t)}{s-t}
\}_{s,t \in \mathbb{R}} \Vert_{\mathfrak{m}_{p}}  \leq \Vert \widehat{G} \Vert_1  < \infty, 
\]
and this concludes the proof. 
\end{proof}

\begin{lemma} \label{Lem=FiniteBandwith}
For $1 \leq p < \infty, R >0$ we have that $\Vert \{  
\rho_R(s-t) 
\}_{s,t \in \mathbb{R}} \Vert_{\mathfrak{m}_{p}} \leq \Vert \widehat{\rho} \Vert_1 < \infty$.  
\end{lemma}
\begin{proof}
$\rho_R$ is Schwartz and so its Fourier transform is integrable. Therefore,  \cite[Proposition 4.2.2.(ii)]{McDonaldSukochev} yields that  $\Vert \{  
\rho_R(s-t) 
\}_{s,t \in \mathbb{R}} \Vert_{\mathfrak{m}_{p}} \leq \Vert \widehat{\rho_R} \Vert_1$.   Further, 
\[
 \widehat{\rho_R}(t) =  (2\pi)^{-\frac{1}{2}} \int_{\mathbb{R}}  \rho(R^{-1} s) e^{ist} ds =  R  (2\pi)^{-\frac{1}{2}} \int_{\mathbb{R}}  \rho( s) e^{is Rt} ds = R \widehat{\rho}( R t). 
\]
Further,   $\Vert R \widehat{\rho}( R \: \cdot \:) \Vert_1 = \Vert \widehat{\rho} \Vert_1$. Thus combining all estimates yields $\Vert \{  
\rho_R(s-t) 
\}_{s,t \in \mathbb{R}} \Vert_{\mathfrak{m}_{p}} \leq \Vert \widehat{\rho} \Vert_1 < \infty$. 

\end{proof}

\subsection{Wavelet estimate: block diagonal part} The next aim   is to give a bound for higher order divided differences of functions that are typical in a wavelet decomposition. We start by estimating such Schur multipliers around the diagonal.

In the proof of Proposition \ref{Prop=WaveletReal} below we wish to use the transformation formulae given in \cite[Lemma 2.3, Theorem 2.7]{PSS-Advances}. For this  it is most efficient to   appeal to the theory of multiple operator integrals (see the monograph \cite{SkripkaTomskova}). We shall only need such multiple operator integrals in the very special situation that the symbol has a Fourier transform that is integrable and the spectral integral is taken with respect to a unitary. The multiple operator integral can then be defined through \eqref{Eqn=MOI} in Proposition \ref{Prop=MOI} below in an elementary way.

Let $\mathcal{A}(\mathbb{T}^{n+1})$ be the set of functions $\psi$ in  $C(\mathbb{T}^{n+1})$ such that for its Fourier transform we have $\widehat{\psi} \in \ell^1(\mathbb{Z}^{n+1})$. $\mathcal{A}(\mathbb{T}^{n+1})$ is also called the Fourier algebra.

\begin{proposition}\label{Prop=MOI}
Let $\psi \in \mathcal{A}(\mathbb{T}^{n+1})$ and let $U \in B(L^2(\mathbb{R}))$ be unitary. For $x_1, \ldots, x_n \in S_2$ we define the $S_2$ convergent sum
\begin{equation}\label{Eqn=MOI} 
T^U_\psi( x_1, \ldots, x_n ) := \sum_{k_0, \ldots, k_n \in \mathbb{Z}} \widehat{\psi}(k_0, \ldots, k_n) U^{k_0} x_1  U^{k_1} x_2 U^{k_2} \ldots  U^{k_{n-1}} x_n U^{k_{n}}. 
\end{equation}
 Let $0 <  p_1, \ldots, p_n < \infty$ and  assume $0 < p := (p_1; \ldots; p_n) \leq 1$.   If moreover $x_i \in   S_{p_i} \cap S_2$ and $\widehat{\psi} \in \ell^{p}(\mathbb{Z}^{n+1})$  then  $T^U_\psi( x_1, \ldots, x_n )  \in S_p$ and 
\[
\Vert T^U_\psi: S_{p_1} \times \ldots \times S_{p_n} \rightarrow S_p \Vert \leq \Vert \widehat{\psi} \Vert_{\ell^{p}(\mathbb{Z}^{n+1})}.
\]
\end{proposition}
\begin{proof}
We first prove the $S_2$ convergence of \eqref{Eqn=MOI}. For any subset $A \subseteq \mathbb{Z}^{n+1}$ we have, 
\[
\begin{split}
& \Vert  \sum_{k_0, \ldots, k_n \in A} \widehat{\psi}(k_0, \ldots, k_n) U^{k_0} x_1  U^{k_1} x_2 U^{k_2} \ldots  U^{k_{n-1}} x_n U^{k_{n}} \Vert_2  \\
\leq & \sum_{k_0, \ldots, k_n \in A}  \vert \widehat{\psi}(k_0, \ldots, k_n) \vert  \Vert  U^{k_0} x_1  U^{k_1} x_2 U^{k_2} \ldots  U^{k_{n-1}} x_n U^{k_{n}} \Vert_2 \\
\leq &  \Vert \widehat{\psi} \Vert_{\ell^{1}(A)}   \sup_{k_0, \ldots, k_n \in A} \Vert  U^{k_0} x_1  U^{k_1} x_2 U^{k_2} \ldots  U^{k_{n-1}} x_n U^{k_{n}} \Vert_2 \\
\leq &  \Vert \widehat{\psi} \Vert_{\ell^{1}(A)} \Vert   x_1 \Vert_{2n}  \Vert  x_2 \Vert_{2n}   \ldots  \Vert x_n   \Vert_{2n}\\
\leq &  \Vert \widehat{\psi} \Vert_{\ell^{1}(A)} \Vert   x_1 \Vert_{2}  \Vert  x_2 \Vert_{2}   \ldots  \Vert x_n   \Vert_{2}.
\end{split}
\]
This estimate assures that \eqref{Eqn=MOI} converges in $S_2$ in case  $\widehat{\psi} \in \ell^1(\mathbb{Z}^{n+1})$ as the infinite sum \eqref{Eqn=MOI} is a Cauchy sum.

 Now suppose moreover that $\widehat{\psi} \in \ell^p(\mathbb{Z}^{n+1})$; as $0 < p \leq 1$ in particular  $\widehat{\psi} \in \ell^1(\mathbb{Z}^{n+1})$. Then for $A \subseteq \mathbb{Z}^{n+1}$ and $x_i \in S_{p_i} \cap S_2, \Vert x_i \Vert_{p_i} \leq 1$ we have by the    quasi-triangle inequality \eqref{Eqn=Quasi-Banach}  and H\"older, 
\begin{equation}\label{Eqn=QuasiEstimate}
\begin{split}
     & \Vert  \sum_{k_0, \ldots, k_n \in A} \widehat{\psi}(k_0, \ldots, k_n) U^{k_0} x_1  U^{k_1} x_2 U^{k_2} \ldots  U^{k_{n-1}} x_n U^{k_{n}} \Vert_p^{p}  \\
\leq & \sum_{k_0, \ldots, k_n \in A}  \vert \widehat{\psi}(k_0, \ldots, k_n) \vert^p  \Vert  U^{k_0} x_1  U^{k_1} x_2 U^{k_2} \ldots  U^{k_{n-1}} x_n U^{k_{n}} \Vert_p^p  \\
\leq &  \sum_{k_0, \ldots, k_n \in A}  \vert \widehat{\psi}(k_0, \ldots, k_n) \vert^p  \\
\leq &  \Vert  \widehat{\psi} \Vert_{\ell^p(A)}^p.
\end{split}
\end{equation}
Again, it follows that the sum \eqref{Eqn=MOI} is a Cauchy sum in  $S_p$ if $\widehat{\psi} \in \ell^p(\mathbb{Z}^{n+1})$ and that $\Vert T_\psi(x_1, \ldots, x_n) \Vert_p \leq \Vert \widehat{\psi} \Vert_{  \ell^p(\mathbb{Z}^{n+1}) }$.   Indeed for $\varepsilon > 0$ there exists $M >  0$ large such that for every $N_1 >  N_2 > M$ we have   
\[
\begin{split}
& \Vert \sum_{\substack{ k_0, \ldots, k_n \in \mathbb{Z}, \\
\vert k_0 \vert, \ldots, \vert k_n \vert \leq N_1}}  \widehat{\psi}(k_0, \ldots, k_n) U^{k_0} x_1  U^{k_1} x_2 U^{k_2} \ldots  U^{k_{n-1}} x_n U^{k_{n}} \\
& \qquad -
\sum_{\substack{ k_0, \ldots, k_n \in \mathbb{Z}, \\
\vert k_0 \vert, \ldots, \vert k_n \vert \leq N_2}}  \widehat{\psi}(k_0, \ldots, k_n) U^{k_0} x_1  U^{k_1} x_2 U^{k_2} \ldots  U^{k_{n-1}} x_n U^{k_{n}}  
\Vert_p \\
\leq & \Vert \widehat{\psi} \Vert_{\ell^p( [-N_1, N_1]^{n+1} \backslash [-N_2, N_2]^{n+1}  )} \leq  \Vert \widehat{\psi} \Vert_{\ell^p( \mathbb{Z}^{n+1} \backslash [-M, M]^{n+1}  )} < \varepsilon. 
 \end{split}
\]
\end{proof}

 \begin{remark}
Evidently the statement of Proposition \ref{Prop=MOI} is also true in case $1 \leq p < \infty, \widehat{\psi} \in \ell^1(\mathbb{Z}^{n+1})$ and consequently $1 \leq p_1, \ldots, p_n < \infty$. The proof is easier as one uses the conventional triangle inequality instead of the quasi-triangle inequality. In fact, many of the statements below have a well known counterpart for $1 \leq p < \infty$; these statements shall not be used however in this paper, essentially as in the range $1 < p <\infty$  we can appeal to  Theorem \ref{Thm=PSS} to estimate Schur multipliers, and we decided not to present them here. 
 \end{remark}

\begin{proposition}\label{Prop=Torus} 
Let $0< p_1, \ldots, p_n < \infty$ and assume  $0 < p := (p_1; \ldots; p_n) \leq 1$.   Let $\beta \in \mathbb{N}$ with $n+1 < \beta p$. Let $\varphi \in C^{\beta}(\mathbb{T})$ and let $U \in B(L^2(\mathbb{R}))$ be unitary.   Then, $\varphi^{[n]} \in \mathfrak{U}_n^c(\mathbb{T})$ and, 
\[
\| T_{\varphi^{[n]}}^U : S_{p_1} \times \cdots \times S_{p_n}  \to S_p  \| < \infty.
\]
\end{proposition}

\begin{proof}
Our aim is to show that $\varphi^{[n]}$ satisfies the criteria of Proposition \ref{Prop=MOI}. 
Consider the  Fourier expansion
\[
\varphi(z) = \sum_{k\in\mathbb{Z}} \widehat{\varphi}(k)\, z^k, \qquad z \in \mathbb{T}.
\]
Thus, as taking divided differences  is a linear operation, we get that
\[
\varphi^{[n]} = \sum_{k\in\mathbb{Z}} \widehat{\varphi}(k)\,( z^k )^{[n] }.
\]
We  examine the term $( z^k )^{[n] }$ and make it concrete.  
For $k \geq n > 0$, the first-order identity
\[
\frac{z^k - w^k}{z-w} = \sum_{l=0}^{k-1} z^l w^{k-l-1}, \qquad z,w \in \mathbb{T}, 
\]
may be iterated $n$ times to yield
\[
(z^k)^{[n]}(z_0,\dots,z_n)
= \sum_{(\alpha_0,\dots,\alpha_n)\in\mathcal{A}_{n,k}}
z_0^{\alpha_0} z_1^{\alpha_1} \cdots z_n^{\alpha_n}, \qquad z_0, \ldots, z_n \in \mathbb{T}, 
\]
where $\mathcal{A}_{n,k}$ is the finite index set of $(n+1)$-tuples $(\alpha_0,\dots,\alpha_n)$ of nonnegative integers such that $\sum_{i=0}^n \alpha_i = k-n$.  We have \cite[Section 1.2]{StarsAndBars}, 
\begin{equation}\label{Eqn=BinomialFormula}
|\mathcal{A}_{n,k}| = \binom{k}{n} \qquad \text{for} \qquad k \ge n+1.
\end{equation}
If $k = n$ we have in particular that $(z^n)^{[n]}(z_0,\dots,z_n) = 1$, i.e. a constant function 1. It thus follows that   $(z^k)^{[n]} = 0$ in case $n > k \geq 0$.

For the negative powers we have for $k > 0, n > 0$,  
\[
\frac{z^{-k} - w^{-k}}{z-w}
=  z^{-k} \frac{w^{k} - z^{k}}{z-w}  w^{-k}   =
- z^{-k} (\sum_{l=0}^{k-1}  z^l w^{k-1-l})   w^{-k}
=
- \sum_{l=0}^{k-1} z^{l-k} w^{-l-1}, \qquad z,w \in \mathbb{T}, 
\]
and applying this formula $n$ times yields 
\[
(z^{-k})^{[n]}(z_0,\dots,z_n)
= (-1)^n \sum_{(\alpha_0,\dots,\alpha_n)\in\mathcal{B}_{n,k}}
z_0^{-\alpha_0-1} z_1^{-\alpha_1-1} \cdots z_n^{-\alpha_n-1},\qquad z_0, \ldots, z_n \in \mathbb{T}, 
\]
where  $\mathcal{B}_{n,k}$  is the finite index set of $(n+1)$-tuples $(\alpha_0,\dots,\alpha_n)$ of nonnegative integers such that $\sum_{i=0}^n \alpha_i = k-1$. We now have  \cite[Section 1.2]{StarsAndBars}, 
\begin{equation}\label{Eqn=BinomialFormula2}
|\mathcal{B}_{n,k}| = \binom{n+k-1}{k-1}. 
\end{equation}

Now note that \eqref{Eqn=BinomialFormula} and \eqref{Eqn=BinomialFormula2} imply that  there exists a constant  $C_n > 0$ depending only on $n$ such that
\[
\max(|\mathcal{A}_{n,k}|, |\mathcal{B}_{n,k}|) \le C_n\,(1+|k|)^n, \qquad k\in\mathbb{Z}.
\]
It follows that the Fourier expansion of $\varphi^{[n]}$ is given by the following formula, where $z_0, \ldots, z_n \in \mathbb{T}$, 
\begin{equation}\label{Eqn=FourierExpansion}
\begin{split}
\varphi^{[n]}(z_0, \ldots, z_n) =  & \sum_{k \in \mathbb{Z}_{\geq n} } \widehat{\varphi}(k) \sum_{(\alpha_0,\dots,\alpha_n)\in\mathcal{A}_{n,k}}  z_0^{\alpha_0} z_1^{\alpha_1} \cdots z_n^{\alpha_n} \\
& \quad + \quad  
\sum_{k \in \mathbb{Z}_{< 0} } \widehat{\varphi}(k) \sum_{(\alpha_0,\dots,\alpha_n)\in\mathcal{B}_{n,\vert k \vert}} 
z_0^{-\alpha_0-1} z_1^{-\alpha_1-1} \cdots z_n^{-\alpha_n-1}. 
\end{split}
\end{equation}
Therefore, by the quasi-triangle inequality, 
\begin{equation}\label{Eqn=DivDiffFourierEst}
\Vert \widehat{ \varphi^{[n]} } \Vert_{\ell^p(\mathbb{Z}^{n+1})}^p =   \sum_{k \in \mathbb{Z}_{\geq n} } \vert \widehat{\varphi}(k) \vert^p   \vert \mathcal{A}_{n,k} \vert  +   
\sum_{k \in \mathbb{Z}_{< n} } \vert \widehat{\varphi}(k) \vert^p   \vert \mathcal{B}_{n,\vert k\vert } \vert
\leq C_n  \sum_{k \in \mathbb{Z} } \vert \widehat{\varphi}(k) \vert^p  \,(1+|k|)^n.
\end{equation}
Now if  $\varphi \in C^\beta(\mathbb{T})$, then by \cite[Prop.~3.2.9 (b)]{GrafakosClassical}
\[
|\widehat{\varphi}(k)| \preceq (1+|k|)^{-\beta}.
\]
And thus by \eqref{Eqn=DivDiffFourierEst} we have $\Vert \widehat{ \varphi^{[n]} } \Vert_{\ell^p(\mathbb{Z}^{n+1})} < \infty$ in case  $n - \beta p < -1$, that is $ \beta p > n+1$. Thus we conclude the proof by Proposition \ref{Prop=MOI}.  
The fact that $\varphi^{[n]} \in \mathfrak{U}_n^{c}(\mathbb{T})$ now also follows by the expansion  \eqref{Eqn=FourierExpansion} and the fact that \eqref{Eqn=DivDiffFourierEst} certainly implies that $\Vert \widehat{ \varphi^{[n]} } \Vert_{\ell^1(\mathbb{Z}^{n+1})} < \infty$ as $p \leq 1$. 

\end{proof}

We now transfer our result to Schur multipliers with symbols depending on real variables. 

\begin{proposition} \label{Prop=WaveletReal}
Let $0< p_1, \ldots, p_n < \infty$ and assume that  $0 < p := (p_1; \ldots; p_n) \leq 1$.   Let $\beta \in \mathbb{N}$ with $n+1 < \beta p$. If $\phi \in C^{\beta}_c(\mathbb{R})$, then
\[
\|  \phi^{[n]}   \|_{\mathfrak{m}_{p_1, \ldots, p_n}} < \infty.
\]
\end{proposition}
\begin{proof}
Recall the Cayley transform $G: \mathbb{T} \backslash \{1 \} \rightarrow \mathbb{R}$ which is a smooth bijection given by 
\[
G(z) =  i \frac{z+1}{z-1}, \qquad  z \in \mathbb{T} \backslash \{1 \}.
\]
$G$ has the property that $\lim_{\lambda \rightarrow \pm \infty }G^{-1}(\lambda) = 1$. 
Set $\varphi(z) = \phi \circ G(z)$ in case $z \in \mathbb{T} \backslash \{1 \}$ and as $\phi$ is compactly supported we may continuously extend $\varphi$ to $\mathbb{T}$ by setting $\varphi(1) = 0$.  As $G$ is smooth we then have $\varphi \in C^{\beta}(\mathbb{T})$.   Proposition \ref{Prop=Torus}  shows that $\varphi^{[n]} \in \mathfrak{U}_n^c(\mathbb{T})$ and thus  $\phi^{[n]} \in \mathfrak{U}_n^c(\mathbb{R})$, by the proof of \cite[Theorem 2.4 (ii)]{PSS-Advances}.

Let $\lambda_i \in \mathbb{R}$ and set $z_i = G^{-1}(\lambda_i)$ where $i = 0, \ldots, n$. 
By \cite[Lemma 2.3 (ii)]{PSS-Advances} we have
\begin{equation}\label{Eqn=CayleyTransformed} 
\begin{split}
\phi^{[n]}(\lambda_0, \ldots, \lambda_n) = & \sum_{k=1}^n \sum_{0=i_0 < \ldots < i_k = n} \frac{(-1)^{k+1}  i^{n-k+1}  }{ 2^{n-k+1}} \varphi^{[k]}( z_{i_0}, \ldots, z_{i_k} ) \\
& \qquad \times \qquad \prod_{j=1}^{k-1}  (z_{i_j} - 1)^2  \prod_{l \in \{ 0, \ldots, n \} \backslash \{ i_1, \ldots, i_{k-1} \} } (z_l -1). 
\end{split}
\end{equation}
 

 Let $H$ be the unbounded self-adjoint operator introduced in Section \ref{Sect=Comparison}.
Let   $U = G^{-1}(H)$ be the multiplication operator on $L^2(\mathbb{R})$ with the function $G^{-1}$. As $G^{-1}$ takes values in $\mathbb{T} \backslash \{1 \}$ we see that $U$ is unitary.
 Proposition \ref{Prop=MOI} then defines the multilinear map  $T_{ \varphi^{[k]} }^{U}, 1 \leq k \leq n$.

We now apply  Lemma \ref{Lem=Comparison} and the transformation formulae \cite[Theorem 2.7,  Lemma 2.2]{PSS-Advances}   to the function \eqref{Eqn=CayleyTransformed}. Hence we find for $x_i \in S_{p_i} \cap S_{2}$, and hence certainly $x_i \in S_{n}$ as in the assumptions of \cite[Theorem 2.7]{PSS-Advances}, that 
  \[
\begin{split}
 T_{\phi^{[n]}}(x_1, \ldots, x_n) = & T_{\phi^{[n]}}^H(x_1, \ldots, x_n) \\    
 = & \sum_{k=1}^n \sum_{0=i_0 < \ldots < i_k = n} \frac{(-1)^{k+1}  i^{n-k+1}  }{ 2^{n-k+1}} 
 T_{ \varphi^{[k]} }^{U}( 
 X_{i_0+1} \ldots X_{i_1}   , \ldots, \\
 & \qquad \qquad\qquad  X_{i_{k-2} + 1}  \ldots  X_{i_{k-1}}   ,  X_{i_{k-1} + 1}  \ldots  X_{i_k}    ), \\
 & 
\end{split}
\]
where
\[
X_{l} =
\left\{
\begin{array}{ll}
x_{l}  (H - 1),  &  l \in \{ 0, \ldots, n \} \backslash \{ i_1, \ldots, i_{k-1} \},  \\
x_l (H - 1)^2,                 &   \textrm{otherwise}. 
\end{array}
\right.
\]
  Note that the H\"older combination of the exponents 
\[
(p_{i_0 + 1}; \ldots; p_{i_1}), (p_{i_1 + 1}; \ldots; p_{i_2}), \ldots, (p_{i_{k-1} + 1}; \ldots; p_{i_k}), \qquad  0 = i_0 < \ldots < i_k = n, 
\]
is given by 
\[
\sum_{l=1}^{k-1} \sum_{s= i_l + 1}^{i_{l+1}} \frac{1}{p_s} = \sum_{l=1}^n \frac{1}{p_l} = \frac{1}{p}. 
\]
By the quasi-triangle inequality and Proposition   \ref{Prop=Torus} we see that the condition $\varphi \in C^\beta(\mathbb{T})$ implies that 
\[
\begin{split}
 & \Vert T_{\phi^{[n]}}(x_1, \ldots, x_n) \Vert_p    \\ \preceq_{n,p} & \sup_{k=1, \ldots, n} \sup_{ \{ 0 = i_0 < \ldots < i_k = n\} \subseteq \{ 0 , \ldots, n\} }  \Vert T_{\varphi^{[k]}}^{U}: S_{(p_{i_0 + 1}; \ldots; p_{i_1})} \times \ldots \times S_{(p_{i_{k-1} + 1}; \ldots; p_{i_k})} \rightarrow S_p \Vert \prod_{l=1}^n \Vert X_l \Vert_{p_l} \\
 \preceq &   \sup_{k=1, \ldots, n} \sup_{ \{ 0 = i_0 < \ldots < i_k = n\} \subseteq \{ 0 , \ldots, n\} }  \Vert T_{\varphi^{[k]}}^{U}: S_{(p_{i_0 + 1}; \ldots; p_{i_1})} \times \ldots \times S_{(p_{i_{k-1} + 1}; \ldots; p_{i_k})} \rightarrow S_p \Vert  \prod_{l=1}^n \Vert x_l \Vert_{p_l},
\end{split}
\]
and the norms of $T_{\varphi^{[k]}}^{U}$ is finite. As the $x_i$ we considered are dense in $S_{p_i}$ the proof follows. 
\end{proof}

The following lemma, as well as its proof, is standard but we have not found its precise statement in the literature. We give the proof for completeness. Note that in the proof of the lemma we must restrict ourselves to the range $1 \leq p \leq \infty$. 

\begin{lemma}\label{Lem=Expect}
Consider two families $\{  Q_{k} \}_{k \in \mathbb{Z}}$ and $\{  P_{k} \}_{k \in \mathbb{Z}}$ of mutually orthogonal projections acting on a Hilbert space $H$. Consider the map 
\[
E: B(H) \rightarrow B(H): x \mapsto \sum_{k \in \mathbb{Z}} Q_k x P_k, 
\]
where the sum converges in the strong operator topology. Let $p \in [1, \infty)$. Then if $x \in S_p$ we have $E(x) \in S_p$ and the assignment $x \mapsto E(x)$ extends to a contraction on $S_p$.
\end{lemma}
\begin{proof}
Let $P = \sum_{k \in \mathbb{Z}} P_k$ and $Q = \sum_{k \in \mathbb{Z}} Q_k$. 
Consider $p=1$. Let $x \in S_1$ and write $x = ab$ with $a, b \in S_2$ and $\Vert x \Vert_1 = \Vert a \Vert_2 \Vert b \Vert_2$; note that we may take $a = u \vert x \vert^{\frac{1}{2}}, b = \vert x \vert^{\frac{1}{2}}$ where $x = u \vert x \vert$ is the polar decomposition. Then 
\[
E(x) \otimes e_{1,1} =   ( \sum_{k \in \mathbb{Z}} Q_k a \otimes e_{1, k} ) ( \sum_{k \in \mathbb{Z}} b P_k \otimes e_{k, 1} ),
\]
where $e_{i,j}$ are the matrix units in $B(\ell^2(\mathbb{Z}))$. Therefore
\[
\begin{split}
\Vert E(x) \Vert_1 = & \Vert E(x) \otimes e_{1,1} \Vert_1 \leq  
\Vert \sum_{k \in \mathbb{Z}} Q_k a \otimes e_{1, k} \Vert_2 \Vert \sum_{k \in \mathbb{Z}} b P_k \otimes e_{k, 1} \Vert_2 \\
= & (\sum_{k \in \mathbb{Z}} \Vert Q_k a \Vert_2^2)^{\frac{1}{2}}  (\sum_{k \in \mathbb{Z}} \Vert b P_k \Vert_2^2)^{\frac{1}{2}}
= \Vert Q a \Vert_2 \Vert b P \Vert_2 \leq  \Vert a \Vert_2  \Vert b \Vert_2 =\Vert x \Vert_1. 
\end{split}
\]
This concludes the proof for $p=1$.

For $p=\infty$ we note that $E$ is the dual of the contractive map
\[
E': S_1 \rightarrow S_1: x \mapsto \sum_{k \in \mathbb{Z}} P_k x Q_k, 
\]
through the duality pairing $\langle x, y \rangle_{S_1, S_\infty} = {\rm Tr}(xy)$. Therefore $E$ is contractive and normal on $B(H)$. 

Finally for general $1 \leq p \leq \infty$ the proof follows from complex interpolation.

\end{proof}

\begin{remark}
In the following Proposition \ref{Prop=Diagonal} for $n=1$ the conditions force that $p=p_1=1$. 
\end{remark}

\begin{proposition}\label{Prop=Diagonal}
Let $1 \leq p_1, \ldots, p_n < \infty$ and assume that $0 < p := (p_1; \ldots; p_n) \leq 1$.  For $0 \leq i \leq n$ let $\{  Q_{i,k} \}_{k \in \mathbb{Z}}$ be a family of mutually orthogonal projections. Let, for $k \in \mathbb{Z}$,
\[
T_k:    S_{p_1}   \times \ldots \times   S_{p_n}   \rightarrow     S_{p}, 
\]
be a multilinear map such that   
\begin{equation}\label{Eqn=BlockType}
T_k(   x_1, \ldots, x_n  ) =   T_k(  Q_{0,k} x_1 Q_{1,k}, Q_{1,k} x_2 Q_{2,k}, \ldots, Q_{n-1,k} x_n Q_{n,k}  ).  
\end{equation}
Then, $T := \sum_{k \in \mathbb{Z}} T_k$ satisfies 
\[
\begin{split}
\Vert  T: S_{p_1} \times \ldots \times S_{p_n} \rightarrow S_{p}  \Vert \leq &    \Vert \{ \Vert T_{k}:  S_{p_1} \times \ldots \times S_{p_n} \rightarrow S_{p}   \Vert \}_{k \in \mathbb{Z}} \Vert_{\ell^{\infty}}.  
\end{split}
\]
\end{proposition}
\begin{proof}
  Let $x_i \in S_{p_i}, 1 \leq i \leq n$. Then, by the quasi-triangle inequality, the property \eqref{Eqn=BlockType}, the definition of the operator norm of $T_k$, and finally  the H\"older inequality for exponents $\infty, p_1, \ldots p_n$,  
  \begin{equation}\label{Eqn=BlockHolder}
\begin{split}
 \Vert   T(x_1, \ldots x_n) \Vert_p   \leq & \left( \sum_{k \in \mathbb{Z}} \Vert T_k(x_1, \ldots, x_n) \Vert_p^p \right)^{\frac{1}{p}} \\
  = & \left( \sum_{k \in \mathbb{Z}} \Vert T_k( Q_{0,k} x_1 Q_{1,k}, \ldots, Q_{n-1,k} x_n Q_{n,k}) \Vert_p^p   \right)^{\frac{1}{p}} \\
  \leq &  \left( \sum_{k \in \mathbb{Z}} \Vert T_k \Vert^p \Vert  Q_{0,k} x_1 Q_{1,k} \Vert_{p_1}^p  \ldots \Vert Q_{n-1,k} x_n Q_{n,k}) \Vert_{p_n}^p \right)^{\frac{1}{p}}  \\
    \leq &  \Vert \{ \Vert T_k \Vert\}_{k \in \mathbb{Z}} \Vert_{\ell^{\infty}} 
   \Vert  \{ \Vert  Q_{0,k} x_1 Q_{1,k} \Vert_{p_1} \}_{k \in \mathbb{Z}} \Vert_{\ell^{p_1}}  \ldots \Vert \{ \Vert Q_{n-1,k} x_n Q_{n,k}) \Vert_{p_n} \}_{k \in \mathbb{Z}}  \Vert_{\ell^{p_n}}.
   \end{split}
  \end{equation}
  Now as $1 \leq p_i < \infty$ we can apply Lemma \ref{Lem=Expect} to get the following inequality,  
  
  \begin{equation}\label{Eqn=DiagonalProjection}
   \Vert \{  \Vert  Q_{i-1,k} x_i Q_{i,k} \Vert_{p_i} \}_{k \in \mathbb{Z}} \Vert_{\ell^{p_i}} = \Vert \sum_{k \in \mathbb{Z}} Q_{i-1,k} x_i Q_{i,k} \Vert_{\ell^{p_i}} \leq \Vert   x_i   \Vert_{p_i}.
  \end{equation}
  The   two estimates \eqref{Eqn=BlockHolder} and \eqref{Eqn=DiagonalProjection} conclude the proposition. 
\end{proof}

Recall that the cut off function $\rho_R$ was defined in \eqref{Eqn=RhoR}. At this point we introduce for $\phi \in C_c(\mathbb{R})$ and $\alpha = \{ \alpha_k \}_{k \in \mathbb{Z}} \in \ell^{\infty}(\mathbb{Z})$ the function 
\begin{equation}\label{Eqn=PhiAlpha}
\phi_{\alpha, \lambda}(t) = \sum_{k \in \mathbb{Z}} \alpha_k \phi(\lambda t - k), \qquad \phi_{\alpha} = \phi_{\alpha, 1}, \qquad t \in \mathbb{R}.
\end{equation} 
We will later take $\phi$ to be a wavelet and $\alpha_k$ to be the coefficients in a wavelet decomposition on which we impose further decay assumptions. For now we have the following. 

\begin{theorem}\label{Thm=DiagonalMain}
Let $1 \leq p_1, \ldots, p_n < \infty$ and assume that  $0 < p := (p_1; \ldots; p_n) \leq 1$.   Let $\beta \in \mathbb{N}$ with $n+1 < \beta p$. Let $\phi \in C^{\beta}_c(\mathbb{R})$.  Then,  there exists a constant $C > 0$ such that for every $\alpha \in \ell^{\infty}$ we have, 
\[
\Vert \{ \phi^{[n]}_{\alpha}(t_0, \ldots, t_n) 
\rho_R(t_0 - t_1) \ldots \rho_R(t_{n-1} - t_n)  \}_{t_0, \ldots, t_n \in \mathbb{R}} \Vert_{\mathfrak{m}_{p_1, \ldots, p_n}}   < C \Vert \alpha \Vert_{\infty}.
\]
\end{theorem}
\begin{proof}
For $r \in \mathbb{Z}$ we introduce the block diagonal  indicator concentrated along the $r$-th off-diagonal,
\[
b_{r}(s,t) = \sum_{k \in \mathbb{Z}} \chi_{[k, k+1)}(s)  \chi_{[k+r, k+r+1)}(t), \qquad s,t \in \mathbb{R}. 
\]
The function $B_R = \sum_{r = -2R-1 }^{2R+1} b_r, R \in \mathbb{N}$ is then again an indicator function whose support is strictly larger than the support of $(s,t) \mapsto \rho_R(s-t)$. Now set for $r_1, \ldots, r_n \in \mathbb{Z}, R_1, \ldots, R_n \in \mathbb{N}$ and $t_0, \ldots, t_n \in \mathbb{R}$,  
\[
\begin{split}
b_{r_1, \ldots, r_n}(t_0, \ldots, t_n) = & b_{r_1}(t_0, t_1) \ldots b_{r_n}(t_{n-1}, t_n),  \\
B_{R_1, \ldots, R_n}(t_0, \ldots, t_n) = & B_{R_1}(t_0, t_1) \ldots B_{R_n}(t_{n-1}, t_n).  \\
\end{split}
\]
Also set,
\[
\begin{split}
b_{r_1, \ldots, r_n,l}(t_0, \ldots, t_n) =   & \chi_{[l, l+1)}(t_0)b_{r_1, \ldots, r_n}(t_0, \ldots, t_n) \\
= &  \chi_{[l, l+1)}(t_0) \chi_{[l+r_1, l+r_1+1)}(t_1) 
\ldots  \chi_{[l+ r_1+ \ldots + r_n, l+r_1+\ldots +r_{n}+1)}(t_n). 
\end{split}
\]
Note that $b_{r_1, \ldots, r_n} = \sum_{l \in \mathbb{Z}} b_{r_1, \ldots, r_n,l}$. 
Then $B_{R, \ldots, R}(t_0, \ldots, t_n)$ is again an indicator function whose support is strictly larger than the support of  
\[
d_R(t_0, \ldots, t_n) := \prod_{i=1}^n  \rho_R(t_{i-1} - t_{i}).
\]
Note that 
\[
T_{ \phi^{[n]}_{\alpha}  b_{r_1, \ldots, r_n} d_R  } = T_{ \phi^{[n]}_{\alpha}  b_{r_1, \ldots, r_n}   }  \circ (T_{ \widetilde{ \rho}_R}  \times \ldots \times T_{ \widetilde{ \rho}_R} ),
\]
where $\widetilde{\rho}_R(s,t) = \rho_R(s-t), s,t \in \mathbb{R}$. Hence,   by  Lemma \ref{Lem=FiniteBandwith},
\[
\Vert   \phi^{[n]}_{\alpha} d_R b_{r_1, \ldots, r_n}   \Vert_{\mathfrak{m}_{p_1, \ldots, p_n}} \leq \Vert   \phi^{[n]}_{\alpha}  b_{r_1, \ldots, r_n}   \Vert_{\mathfrak{m}_{p_1, \ldots, p_n}} \prod_{i=1}^n \Vert  \widetilde{ \rho}_R \Vert_{\mathfrak{m}_{p_i}} \preceq \Vert   \phi^{[n]}_{\alpha}  b_{r_1, \ldots, r_n}   \Vert_{\mathfrak{m}_{p_1, \ldots, p_n}}.
\]
By the quasi-triangle inequality,   we thus get
\[
\begin{split}
\Vert   \phi^{[n]}_{\alpha} d_R   \Vert_{\mathfrak{m}_{p_1, \ldots, p_n}}^p   = \Vert   \phi^{[n]}_{\alpha} d_R B_{R, \ldots, R}   \Vert_{\mathfrak{m}_{p_1, \ldots, p_n}}^p & \leq     \sum_{r_1, \ldots, r_n = -2R-1}^{2R+1} \Vert   \phi^{[n]}_{\alpha} d_R b_{r_1, \ldots, r_n}   \Vert_{\mathfrak{m}_{p_1, \ldots, p_n}}^p \\
&  \preceq    \sum_{r_1, \ldots, r_n = -2R-1}^{2R+1} \Vert   \phi^{[n]}_{\alpha}  b_{r_1, \ldots, r_n}   \Vert_{\mathfrak{m}_{p_1, \ldots, p_n}}^p. 
\end{split}
\]
This summation is finite and hence we have reduced the problem to showing that each of the individual summands is bounded up to a constant by  $\Vert \alpha \Vert_{\infty}^p$. 

For a function $\psi \in L^\infty(\mathbb{R}^{n+1})$ and $k \in \mathbb{R}$ we set the translated function, 
\[
(\tau_k \psi)(t_0, \ldots, t_n ) = \psi(t_0 - k, \ldots, t_n - k).
\]
As the divided difference of a translated function equals the translation of the divided difference (see \eqref{Eqn=DivDiffDef}), we have that 
\begin{equation}\label{Eqn=PhiShifted}
\phi^{[n]}_{\alpha}(t_0, \ldots, t_n) = \sum_{k \in \mathbb{Z}} \alpha_k  \phi^{[n]}(t_0 - k, \ldots, t_n - k) = \sum_{k \in \mathbb{Z}}  \alpha_k ( \tau_k  \phi^{[n]})(t_0, \ldots, t_n).  
\end{equation}
Note that if for all $0 \leq i \leq n$ we have that $t_i-k$ is outside of the support of $\phi$   then the  summand $\alpha_k ( \tau_k  \phi^{[n]})(t_0, \ldots, t_n)$ in  \eqref{Eqn=PhiShifted} is 0; this can be verified inductively by the definition of $n$-th order divided difference functions in terms of the $n-1$-th order divided difference function \eqref{Eqn=DivDiffDef}.  As $\phi$ has compact support it follows that for any $r_1, \ldots, r_n \in \mathbb{Z}$ there exists $N \in \mathbb{N}$ such that for all $k \in \mathbb{Z}, \vert k \vert > N$ we have
\[
  ( \tau_k  \phi^{[n]})  b_{r_1, \ldots, r_n, 0}   = 0.
\]
More precisely, if $\supp(\phi) \subseteq [-K, K]$ then $N = K + \max(\vert r_1 \vert, \ldots, \vert r_n \vert) + 1$ suffices. Hence also for  $k \in \mathbb{Z}, \vert k \vert > N$ and any $l \in \mathbb{Z}$ we have 
\[
 (\tau_{k+l} \phi^{[n]}) b_{r_1, \ldots, r_n, l}  =  \tau_l((\tau_{k} \phi^{[n]}) b_{r_1, \ldots, r_n, 0})  =    0.
\]
We emphasize that $N$ is independent of $l$ and depends on $r_1, \ldots, r_n$ and $\phi$ only. 
It follows that, 
\[
 \phi^{[n]}_{\alpha}  b_{r_1, \ldots, r_n,l} = \sum_{k \in \mathbb{Z}} \alpha_k ( \tau_k  \phi^{[n]} )  b_{r_1, \ldots, r_n, l} = \sum_{k = - N}^N \alpha_{k+l}   ( \tau_{k+l}  \phi^{[n]} )  b_{r_1, \ldots, r_n, l}.
\] 
And thus, 
\begin{equation}\label{Eqn=StartEqn}
 \phi^{[n]}_{\alpha}  b_{r_1, \ldots, r_n} = \sum_{l \in \mathbb{Z}}  \phi^{[n]}_{\alpha}  b_{r_1, \ldots, r_n, l} = \sum_{k = - N}^N  \sum_{l \in \mathbb{Z}}  \alpha_{k+l} (\tau_{k+l}  \phi^{[n]} )  b_{r_1, \ldots, r_n, l}.
\end{equation}
Let $r_0 = 0$. Set $Q_{i,l}, i =0, \ldots, n, l \in \mathbb{Z}$, to be the projection given by the multiplication operator with indicator function $\chi_{[l+k+r_i,l+k+r_i+1)}$. Then, the symbol $ b_{r_1, \ldots, r_n, l}$ ensures that for $x_i \in S_{p_i}$ we have, 
\[
\begin{split}
T_{\tau_{k+l}  \phi^{[n]}   b_{r_1, \ldots, r_n, l}}(x_1, \ldots, x_n)= &
 T_{\tau_{k+l}  \phi^{[n]}  b_{r_1, \ldots, r_n, l}}( Q_{0,l}  x_1 Q_{1,l}, \ldots, Q_{n-1,l} x_n Q_{n,l}) \\
 = & T_{\tau_{k+l}  \phi^{[n]}   }( Q_{0,l}  x_1 Q_{1,l}, \ldots, Q_{n-1,l} x_n Q_{n,l}). \\
\end{split}
\]
We conclude firstly that  we may apply Proposition \ref{Prop=Diagonal} (our index $l$ plays the role of $k$) and secondly we have, 
\begin{equation}\label{Eqn=EstimateThis} 
\begin{split}
\Vert T_{\tau_{k+l}  \phi^{[n]}   b_{r_1, \ldots, r_n, l}}(x_1, \ldots, x_n)  \Vert_p = &
\Vert  T_{\tau_{k+l}  \phi^{[n]}  }( Q_{0,l}  x_1 Q_{1,l}, \ldots, Q_{n-1,l} x_n Q_{n,l})  \Vert_p \\
\leq & \Vert  \tau_{k+l}  \phi^{[n]}  \Vert_{\mathfrak{m}_{p_1, \ldots, p_n}} \Vert Q_{0,l}  x_1 Q_{1,l} \Vert_{p_1} \ldots \Vert Q_{n-1,l} x_n Q_{n,l} \Vert_{p_n} \\
\leq &\Vert  \tau_{k+l}  \phi^{[n]}  \Vert_{\mathfrak{m}_{p_1, \ldots, p_n}} \Vert   x_1   \Vert_{p_1} \ldots \Vert x_n \Vert_{p_n}. \\
\end{split}
   \end{equation}
   Therefore,  by using \eqref{Eqn=StartEqn}, the quasi-triangle inequality,  and   Proposition \ref{Prop=Diagonal} for the first inequality, and \eqref{Eqn=EstimateThis} for the second equality,  
\[
\begin{split}
\Vert   \phi^{[n]}_{\alpha}  b_{r_1, \ldots, r_n}   \Vert_{\mathfrak{m}_{p_1, \ldots, p_n}} \preceq &  \sup_{l \in \mathbb{Z}, \vert k \vert \leq N} \Vert \alpha_{k+l} ( \tau_{k+l}  \phi^{[n]} )  b_{r_1, \ldots, r_n, l}  \Vert_{\mathfrak{m}_{p_1, \ldots, p_n}}      \\
\leq&
\Vert \alpha \Vert_\infty \sup_{l \in \mathbb{Z}, \vert k \vert \leq N}  \Vert ( \tau_{k+l}  \phi^{[n]} )  b_{r_1, \ldots, r_n, l}  \Vert_{\mathfrak{m}_{p_1, \ldots, p_n}}   \\
\leq&
\Vert \alpha \Vert_\infty    \Vert   \tau_{k+l}  \phi^{[n]}  \Vert_{\mathfrak{m}_{p_1, \ldots, p_n}} \\
\leq & \Vert \alpha \Vert_\infty    \Vert   \phi^{[n]}  \Vert_{\mathfrak{m}_{p_1, \ldots, p_n}}. 
\end{split}
\]
Applying Proposition \ref{Prop=WaveletReal} concludes the proof. 
\end{proof}

\subsection{Induction} We are now in a position to prove the main estimate on Schur multipliers of higher order divided differences of a wavelet.  The proof proceeds by induction to the order.  
As part of our proof we need the linear case that was covered in \cite{McDonaldSukochev} and which we recall here. Recall that for $0 < p \leq 1$ we set
\[
p^{\sharp} = \frac{p}{1-p},
\]
where if $p=1$ we set $p^\sharp = \infty$. 
Then $\frac{1}{p^\sharp} + 1 = \frac{1}{p}$ or in other words $p = (p^\sharp; 1)$. Recall that $\phi_{\alpha, \lambda}$ was defined in \eqref{Eqn=PhiAlpha}.

\begin{theorem}[Theorem 4.3.2 of \cite{McDonaldSukochev}]\label{Thm=McDonaldSukochev}
Let  $0 <  p \leq 1$. Let $\phi \in C_c^{\beta}(\mathbb{R})$ with $\beta > \frac{2}{p}$.   There exists a constant $C > 0$ such that for every $\alpha \in \ell^{p^\sharp}$ and $\lambda>0$, we have, 
\begin{equation}\label{Eqn=McDSEstimate}
\begin{split}
\Vert  \phi_{\alpha, \lambda}^{[1]}   \Vert_{\mathfrak{m}_{p }} \leq C \lambda  \Vert \alpha \Vert_{p^\sharp}. 
\end{split}
\end{equation}
\end{theorem}

The following theorem is a direct consequence of a result first proved in \cite{PS-Acta}. After that,   alternative proofs and sharpenings of this statement appeared in \cite{CMPS-JFA, CPSZ-AJM, CJSZ-JFA, CGPT-Annals, GPPR}. We note that at $p=1$ we have $p^\sharp = \infty$ and so the estimates \eqref{Eqn=McDSEstimate} and \eqref{Eqn=PSActa}  agree, though they are stated under different regularity conditions on $\phi$.

\begin{theorem}[\cite{PS-Acta}]
Let  $1 <  p < \infty$. Let $\phi \in C_c^{1}(\mathbb{R})$.  There exists a constant $C > 0$ such that for every $\alpha \in \ell^{\infty}(\mathbb{Z})$ and $\lambda>0$ we have, 
\begin{equation}\label{Eqn=PSActa}
\begin{split}
\Vert  \phi_{\alpha, \lambda}^{[1]}   \Vert_{\mathfrak{m}_{p }} \leq C \lambda  \Vert \alpha \Vert_{\infty}. 
\end{split}
\end{equation}
\end{theorem}
\begin{proof}
Fix  $1 <  p < \infty$. The main result of \cite{PS-Acta} yields that there exists a constant $C>0$ such that for every $\phi, \alpha, \lambda$,
\[
\Vert  \phi_{\alpha, \lambda}^{[1]}   \Vert_{\mathfrak{m}_{p }} \leq  C   \Vert  \phi_{\alpha, \lambda}' \Vert_{\infty}. 
\]
By the chain rule for differentiation $\Vert  \phi_{\alpha, \lambda}' \Vert_{\infty} = \lambda \Vert  \phi_{\alpha, 1}' \Vert_{\infty}$. Further, for $s \in \mathbb{R}$,
\[
\begin{split}
\vert \phi_{\alpha,1}'(s) \vert  \leq & 
  \sum_{k \in \mathbb{Z}, s - k \in \supp(\phi)} \vert \alpha_k \vert \vert \phi'( s - k)\vert \preceq_\phi   \sum_{k \in \mathbb{Z}, s - k \in \supp(\phi)} \Vert \alpha \Vert_\infty \Vert \phi' \Vert_\infty,  
  \end{split}
\]
and as $\phi$ has compact support the sum is finite with a bound on the number of summands that is uniform in $s$.
 All the previous estimates together yield $\Vert  \phi_{\alpha, \lambda}^{[1]}   \Vert_{\mathfrak{m}_{p }} \preceq_\phi   \lambda  \Vert \alpha \Vert_{\infty}$ and we are done. 
\end{proof}

\begin{theorem}\label{Thm=InductiveStep}
Let $n \geq 2$.  Let $1 \leq p_1, \ldots, p_n < \infty(\mathbb{Z})$ and assume that 
\[
0 < p := (p_1; \ldots; p_n) \leq 1, \quad    1 \leq (p_2; \ldots; p_n) < \infty, \quad  1 \leq (p_1; \ldots; p_{n-1}) < \infty.
\]
Let $\beta \in \mathbb{N}$ with $n+1 < \beta p$. Let $\phi \in C^{\beta}_c(\mathbb{R})$.   There exists a  constant $C > 0$ such that for every $\alpha \in \ell^{\infty}$ we have, 
\[
\begin{split}
\Vert  \phi_\alpha^{[n]}   \Vert_{\mathfrak{m}_{p_1, \ldots, p_n}} \leq &  C ( \max_{k=1, \ldots, n-1}(  \Vert  \phi_\alpha^{[n-1]}   \Vert_{\mathfrak{m}_{p_1, \ldots, p_{k-1}, (p_k; p_{k+1}), p_{k+2}, \ldots,  p_n}}       +  \Vert \phi_\alpha^{[n-1]} \Vert_{\mathfrak{m}_{p_1, \ldots,   p_{n-1}}}    \\
& \qquad + \qquad   \Vert  \phi_\alpha^{[n-1]} \Vert_{\mathfrak{m}_{p_2, \ldots,   p_{n}}} +  \Vert \alpha \Vert_{\infty}). 
\end{split}
\]
\end{theorem}
\begin{proof} 
In the next decomposition the hat indicates that a variable is omitted. We have, 
\[
\begin{split}
\phi^{[n]}_\alpha(t_0, \ldots, t_n) = & \phi^{[n]}_\alpha(t_0, \ldots, t_n)  (   \sum_{k=0}^{n-1} \rho_R(   t_0   -   t_1 )   \ldots\rho_R(   t_{k-1}   -   t_k ) \cdot  (1-\rho_R)(   t_{k}   -   t_{k+1} )    \\
& 
 + \quad \rho_R(   t_0   -   t_1 )   \ldots\rho_R(   t_{n-1}   -   t_n )  
)\\
 = &    \sum_{k=0}^{n-1}\left( (\phi_\alpha^{[n-1]}(t_0, \ldots, \widehat{t_k}, \ldots,  t_n)  - \phi^{[n-1]}_\alpha(t_0, \ldots, \widehat{t_{k+1}}, \ldots,  t_n)   ) \right. \\
 & \left. \qquad \times \qquad  \rho_R(   t_0   -   t_1 )   \ldots\rho_R(   t_{k-1}   -   t_k )  \frac{   (1-\rho_R)(   t_{k}   -   t_{k+1} )   }{t_{k+1} - t_k } \right)\\
&  
 + \quad  \phi^{[n]}_\alpha(t_0, \ldots, t_n)\rho_R(   t_0   -   t_1 )   \ldots\rho_R(   t_{n-1}   -   t_n ),   
\end{split}
\] 
By the quasi-triangle inequality it suffices to estimate the multipliers with symbols of each of the $n+1$ summands in the latter expression. By Theorem \ref{Thm=DiagonalMain},
\begin{equation}\label{Eqn=Diagonal}
\Vert \{ \phi^{[n]}_\alpha(t_0, \ldots, t_n) \rho_R(   t_0   -   t_1 )   \ldots\rho_R(   t_{n-1}   -   t_n ) \}_{t_0, \ldots, t_n \in \mathbb{R}} \Vert_{\mathfrak{m}_{p_1, \ldots, p_n}}   \preceq \Vert \alpha \Vert_\infty.
\end{equation}
Further, we have for $0 < k \leq  n-1$, by Lemma \ref{Lem=OneMinRho} and Lemma \ref{Lem=FiniteBandwith},
\begin{equation}\label{Eqn=Induction}
\begin{split}
   &  \Vert \{  \phi_\alpha^{[n-1]}(t_0, \ldots, \widehat{t_k}, \ldots,  t_n)   
\rho_R(   t_0   -   t_1 )   \ldots \rho_R(   t_{k-1}   -   t_k )  \frac{   (1-\rho_R)(   t_{k}   -   t_{k+1} )   }{t_{k+1} - t_k }   \}_{t_0, \ldots, t_n \in \mathbb{R}} \Vert_{\mathfrak{m}_{p_1, \ldots, p_n}} \\
 \leq &   \Vert \{ \phi_\alpha^{[n-1]}(t_0, \ldots, \widehat{t_k}, \ldots,  t_n)   \}_{t_0, \ldots, \widehat{t_k}, \ldots, t_n \in \mathbb{R}} \Vert_{\mathfrak{m}_{p_1, \ldots, (p_{k}; p_{k+1}), \ldots , p_n}}  \Vert \{\frac{   (1-\rho_R)(   t_{k}   -   t_{k+1} )   }{t_{k+1} - t_k }   \}_{t_k, t_{k+1} \in \mathbb{R} } \Vert_{\mathfrak{m}_{p_{k+1}} } \\
 & \quad \times \quad \Vert  \{   \rho_R(   t_0   -   t_1 )   \}_{t_0, t_{1} \in \mathbb{R} } \Vert_{\mathfrak{m}_{p_1} } \cdot \ldots \cdot \Vert  \{  \rho_R(   t_{k-1}  -   t_k)  \}_{t_{k-1}, t_{k} \in \mathbb{R} } \Vert_{\mathfrak{m}_{p_{k}} } \\
 \preceq &   \Vert \{ \phi_\alpha^{[n-1]}(t_0, \ldots, \widehat{t_k}, \ldots,  t_n)   \}_{t_0, \ldots, \widehat{t_k}, \ldots, t_n \in \mathbb{R}} \Vert_{\mathfrak{m}_{p_1, \ldots, (p_{k-1}; p_{k}), \ldots , p_n}}.   
\end{split}
\end{equation}
  For $k=0$ a similar estimate yields the following where the  final estimate is   Theorem \ref{Thm=PSS}, 
\begin{equation}\label{Eqn=PSSPart}
\begin{split}
& \Vert \{     \phi_\alpha^{[n-1]}(t_1, \ldots,   t_n)     \frac{   (1-\rho_R)(   t_{k}   -   t_{k+1} )   }{t_{k+1} - t_k }       \}_{t_0, \ldots, t_n \in \mathbb{R}} \Vert_{\mathfrak{m}_{p_1, \ldots, p_n}} \\
\preceq &  \Vert \{ \phi_\alpha^{[n-1]}(  t_1, \ldots,   t_n)   \}_{t_1, \ldots,   t_n \in \mathbb{R}} \Vert_{\mathfrak{m}_{p_2, \ldots,   p_n}} \preceq \Vert \alpha \Vert_\infty.   
\end{split}
\end{equation} 
In the same way we have for $0 \leq k < n-1$,
\begin{equation} \label{Eqn=Nineteen}
\begin{split}
   &  \Vert \{  \phi_\alpha^{[n-1]}(t_0, \ldots, \widehat{t_{k+1}}, \ldots,  t_n)   
\rho_R(   t_0   -   t_1 )   \ldots \rho_R(   t_{k-1}   -   t_k )  \frac{   (1-\rho_R)(   t_{k}   -   t_{k+1} )   }{t_{k+1} - t_k }   \}_{t_0, \ldots, t_n \in \mathbb{R}} \Vert_{\mathfrak{m}_{p_1, \ldots, p_n}} \\
 \preceq &   \Vert \{ \phi_\alpha^{[n-1]}(t_0, \ldots, \widehat{t_{k+1}}, \ldots,  t_n)   \}_{t_0, \ldots, \widehat{t_{k+1}}, \ldots, t_n \in \mathbb{R}} \Vert_{\mathfrak{m}_{p_1, \ldots, (p_{k}; p_{k+1}), \ldots , p_n}},  
\end{split}
\end{equation}
and for $k = n -1$ using a similar estimate and again Theorem \ref{Thm=PSS},
\begin{equation} \label{Eqn=Twenty}
\begin{split}
& \Vert \{     \phi^{[n-1]}_\alpha(t_0, \ldots,   t_{n-1})   
\rho_R(   t_0   -   t_1 )   \ldots \rho_R(   t_{k-1}   -   t_k )  \frac{   (1-\rho_R)(   t_{k}   -   t_{k+1} )   }{t_{k+1} - t_k }       \}_{t_0, \ldots, t_n \in \mathbb{R}} \Vert_{\mathfrak{m}_{p_1, \ldots, p_n}} \\
\preceq &  \Vert \{ \phi^{[n-1]}_\alpha(  t_0, \ldots,   t_{n-1})   \}_{t_1, \ldots,   t_n \in \mathbb{R}} \Vert_{\mathfrak{m}_{p_1, \ldots,   p_{n-1}}} \preceq \Vert \alpha \Vert_\infty.  
\end{split}
\end{equation} 
 The estimates \eqref{Eqn=Diagonal}, \eqref{Eqn=Induction}, \eqref{Eqn=PSSPart}, \eqref{Eqn=Nineteen} and \eqref{Eqn=Twenty} thus conclude the proof. 
\end{proof}

We now come to our main estimate.

\begin{theorem}\label{Thm=MainWaveletEstimate}
  Let $1 \leq p_1, \ldots, p_n < \infty$ be such that $1 \leq (p_2; \ldots ; p_n), (p_1; \ldots ; p_{n-1}) < \infty$   and let $p := (p_1; \ldots ; p_n)$. Let $\beta \in \mathbb{N}$ with $n+1 < \beta p$ and let $\phi \in C^{\beta}_c(\mathbb{R})$.  
  There exists a constant $C > 0$ such that for every $\alpha \in \ell^{p^\sharp}(\mathbb{Z})$ we have, 
\[
\Vert  \phi_\alpha^{[n]}   \Vert_{\mathfrak{m}_{p_1, \ldots, p_n}} \leq C \Vert \alpha \Vert_{\ell^{p^\sharp}}. 
\]
\end{theorem}
\begin{proof}
The condition  $1 \leq (p_2; \ldots ; p_n), (p_1; \ldots ; p_{n-1}) < \infty$ implies that all of  the H\"older combinations $(p_k; \ldots; p_l), k < l, (k,l) \not = (1,n)$ lie in the interval $[1, \infty)$.  This allows us to inductively apply the estimate of  Theorem \ref{Thm=InductiveStep} to yield 
\begin{equation}\label{Eqn=StartInduction}
\Vert  \phi_\alpha^{[n]}   \Vert_{\mathfrak{m}_{p_1, \ldots, p_n}}   \preceq_n \sup_{1 \leq k < l   \leq n} \Vert  \phi_\alpha^{[1]}   \Vert_{\mathfrak{m}_{(p_k; \ldots; p_l) } } + \Vert \alpha \Vert_\infty.  
 \end{equation}
In case   $(p_k; \ldots; p_l) > 1$  we apply  Theorem \ref{Thm=PSS} and see,
 \begin{equation}\label{Eqn=InductivePSS}
 \Vert  \phi_\alpha^{[1]}   \Vert_{\mathfrak{m}_{(p_k; \ldots; p_l) } }   \preceq \Vert \alpha \Vert_\infty. 
 \end{equation}
In case   $(p_k; \ldots; p_l) \leq 1$  we apply  Theorem \ref{Thm=McDonaldSukochev} and see
 \begin{equation}\label{Eqn=InductiveMcDonaldSukochev}
 \Vert  \phi_\alpha^{[1]}   \Vert_{\mathfrak{m}_{(p_k; \ldots; p_l) } }    \preceq \Vert \alpha \Vert_{p^\sharp}. 
 \end{equation}
 As  $\Vert \alpha \Vert_\infty \leq \Vert \alpha \Vert_{p^\sharp}$ the estimates \eqref{Eqn=StartInduction}, \eqref{Eqn=InductivePSS} and \eqref{Eqn=InductiveMcDonaldSukochev} conclude the proof. 
 \end{proof}

In order to apply Theorem  \ref{Thm=MainWaveletEstimate} to wavelets we shall also need to consider dilations of the symbol. This can be done through a standard argument that we present now.   Recall again that $\phi_{ \alpha, \lambda}$ was defined in \eqref{Eqn=PhiAlpha}.

\begin{lemma}\label{Lem=TranslateWave}
   We have for $\phi \in C_c^n(\mathbb{R})$ and $\alpha \in \ell^\infty(\mathbb{Z}), \lambda > 0$, 
   \[
   \Vert  \phi_{ \alpha, \lambda}^{[n]}   \Vert_{\mathfrak{m}_{p_1, \ldots, p_n}} =
\lambda^n \Vert   \phi_{  \alpha}^{[n]}   \Vert_{\mathfrak{m}_{p_1, \ldots, p_n}}.
   \]

\end{lemma}
\begin{proof}
For a function $\psi \in C^1_c(\mathbb{R})$ we set $\psi_\lambda(t) = \psi(\lambda t), t \in \mathbb{R}$. Then, for $s, t \in \mathbb{R}$, 
\[
\psi_\lambda^{[1]}(s,t) = \frac{\psi(\lambda s) - \psi(\lambda t)}{ s - t} = \lambda \frac{\psi(\lambda s) - \psi(\lambda t)}{ \lambda s - \lambda t} = \lambda \psi^{[1]}( \lambda s, \lambda t). 
\]
By applying this formula inductively to the order $n$ we find that  
\[
 \phi_{\alpha, \lambda}^{[n]}(t_0, \ldots, t_n) = \lambda^n  \phi_{  \alpha}^{[n]}(\lambda t_0, \ldots, \lambda  t_n). 
\]
Now the map $U: L^2(\mathbb{R}) \rightarrow  L^2(\mathbb{R})$ given by $(Uf)(t) = \lambda^{-\frac{1}{2}} f(\lambda t)$ is unitary with $(U^\ast f)(t) = \lambda^{\frac{1}{2}} f(\lambda^{-1} t)$. Further, for $x_1, \ldots, x_n \in S_2$ we have,  
\[
U T_{\phi_{ \alpha, \lambda}^{[n]}}(  x_1 , \ldots,  x_n    ) U^\ast =  \lambda^n T_{\phi_{ \alpha}^{[n]}}(U x_1 U^\ast, \ldots, U x_n U^\ast  ),
\]
so that, by density of $S_2$ in $S_{p_i}$, 
\[
\Vert  \phi_{\alpha, \lambda}^{[n]}   \Vert_{\mathfrak{m}_{p_1, \ldots, p_n}} =
\lambda^n \Vert   \phi_{  \alpha}^{[n]}   \Vert_{\mathfrak{m}_{p_1, \ldots, p_n}}.
\]

\end{proof}

\begin{corollary} \label{Cor=MainWaveletEstimate}
   Using the notation of Theorem \ref{Thm=MainWaveletEstimate}.  There exists a constant $C > 0$ such that for every $\alpha \in \ell^{p^\sharp}(\mathbb{Z})$ and $\lambda > 0$ we have, 
\[
\Vert  \phi_{\lambda, \alpha}^{[n]}   \Vert_{\mathfrak{m}_{p_1, \ldots, p_n}} \leq C \lambda^n \Vert \alpha \Vert_{\ell^{p^\sharp}}. 
\]
\end{corollary}

\section{Wavelet decomposition and main result}\label{Sect=MainResult}

In this section we derive the main result of this paper: an estimate for Schur multipliers of divided diffence functions. This conclusion is derived from the core estimates in Section \ref{Sect=CoreEstimate}. The methods in this section are then similar to \cite[Section 4]{McDonaldSukochev} but we present them in the multilinear case. 

 \begin{definition}
 We call a function $\phi \in L^2(\mathbb{R})$ a wavelet if the family
 \[
 \phi_{j,k}(t) = 2^{\frac{j}{2}} \phi(2^j t - k), \qquad j,k \in \mathbb{Z}, t \in \mathbb{R},
 \]
forms an orthonormal basis in $L^2(\mathbb{R})$.
\end{definition}

\begin{remark}
 In \cite{Daubechies} Daubechies proved that there exists a compactly supported wavelet in $C_c^\beta(\mathbb{R})$ for every $\beta \in \mathbb{N}$ (see also \cite[Theorem 3.8.3]{MeyerBook}).
\end{remark}

We shall make use the following characterization of homogeneous Besov spaces in terms of wavelets. 

\begin{theorem}[Theorem 4.1.3 of \cite{McDonaldSukochev}]
\label{Thm=Besov}
Let $p,q \in (0, \infty]$ and let $s \in \mathbb{R}$. Let $f$ be a locally integrable function and let $\phi$ be a compactly supported $C^\beta$ wavelet for $\beta > \vert s \vert$. Then $f$ belongs to the homogeneous Besov class $\Besov^{s}_{p,q}(\mathbb{R})$ if and only if 
\[
\Vert f \Vert_{\Besov^{s}_{p,q}} \approx_{p,q,s,\phi} \left( \sum_{j \in \mathbb{Z}}  2^{jsq} \Vert f_j \Vert_p^q \right)^{\frac{1}{q}} < \infty. 
\]
\end{theorem}

Now throughout this section we take over the notation from Theorem \ref{Thm=MainWaveletEstimate}.  We let 
\[
1 \leq p_1, \ldots, p_n < \infty,
\]
be such that 
\[
1 \leq (p_2; \ldots ; p_n), (p_1; \ldots ; p_{n-1}) < \infty,
\]
and let $p := (p_1; \ldots ; p_n)$. Let $\beta \in \mathbb{N}$ with $n+1 < \beta p$ and let $\phi \in C^{\beta}_c(\mathbb{R})$ be a compactly supported $C^\beta$-wavelet.

We further let $f: \mathbb{R} \rightarrow \mathbb{R}$ be a locally integrable function. As $\phi$ has compact support we may set 
\begin{equation}\label{Eqn=WaveFj}
f_j := \sum_{k \in \mathbb{Z}} \phi_{j,k} \langle f, \phi_{j,k} \rangle,
\end{equation}
where the sum is finite on compact sets.  If moreover, $f \in L^2(\mathbb{R})$ then as the wavelet yields an orthonormal basis $\{ \phi_{j,k} \}_{j,k}$, we see that the sum \eqref{Eqn=WaveFj} converges in $L^2(\mathbb{R})$ to the function $f_j$ and moreover $f = \sum_{j \in \mathbb{Z}} f_j$ in $L^2(\mathbb{R})$. However, we shall need to apply a wavelet decomposition to functions that are not necessarily in $L^2(\mathbb{R})$ but also need to cover more general  functions with bounded $n$-th order derivative, including polynomials of degree $\leq n$. Such polynomials have the property that the wavelet coefficients are all 0 (see \cite{MeyerBook}) and thus \eqref{Eqn=WaveFj} does not provide a good approximation.  In  the linear case this subtle point was outlined carefully in \cite[Section 4.1]{McDonaldSukochev}. Here we treat the higher order case and it turns out that the degree of regularity that we need precisely coincides with our estimates in Section \ref{Sect=CoreEstimate}. The proof of the following lemma is a straightforward generalisation of \cite[Lemma 4.1.4]{McDonaldSukochev}. 

\begin{lemma}\label{Lem=BesovApprox}
Let $f \in C^n(\mathbb{R}) \cap \Besov_{p^\sharp, p}^{n-1 + \frac{1}{p}}$ with $0 < p \leq 1$ and assume that $\Vert f^{(n)} \Vert_\infty < \infty$. Then there exists a polynomial $P$ of degree at most $n$ such that 
\[
f(t) = P(t) + \sum_{j \in \mathbb{Z}} (f_j(t) -      \sum_{k=0}^{n-1}  \frac{t^k}{k!} f_j^{(k)}(0)), \qquad t \in \mathbb{R}.  
\]
Further,  $\Vert P^{(n)} \Vert_\infty \preceq   \Vert f^{(n)} \Vert_\infty + \Vert f  \Vert_{\Besov_{p^\sharp, p}^{n-1 + \frac{1}{p}}}$ up to a constant that does not depend on $f$.
\end{lemma}
\begin{proof}
   Since the wavelet $\phi$ is assumed to be $C^\beta$ with $n+1 < \beta p$ it follows that $f_j$ is  $\beta$ times continuously differentiable and so certainly it is $n$ times continuously differentiable.   
By \cite[Chapter 2, Theorem 3]{MeyerBook} and then using \cite[Eqn.  (4.3)]{McDonaldSukochev} for every $j \in \mathbb{Z}$ we have 
\[
\Vert f_j^{(n)} \Vert_\infty \preceq 2^{jn} \Vert f_j \Vert_\infty \preceq 2^{jn} 2^{j (\frac{1}{p} - 1)} \Vert f_j \Vert_{p^\sharp}. 
\]
Therefore, we have by Theorem \ref{Thm=Besov} ($s = n -1 + \frac{1}{p}$ and $p \beta > n+1$ imply  $\beta > \vert s \vert$),   and the fact that $0 < p \leq 1$ implies that the norm of  $\Besov^{n-1 + \frac{1}{p}}_{p^\sharp, p}$ is smaller than the norm of $\Besov^{n-1 + \frac{1}{p}}_{p^\sharp, 1}$ (see \eqref{Eqn=BesovDfn}),
\begin{equation}\label{Eqn=DerivativeFinite}
\sum_{j \in \mathbb{Z}}  \Vert f_j^{(n)} \Vert_\infty \preceq 
   \sum_{j \in \mathbb{Z}}   2^{j (n - 1 + \frac{1}{p}  )} \Vert f_j \Vert_{p^\sharp} \approx \Vert f \Vert_{\Besov^{n-1 + \frac{1}{p}}_{p^\sharp, 1} } \leq \Vert f \Vert_{\Besov^{n-1 + \frac{1}{p}}_{p^\sharp, p} }.
\end{equation}
   The convergence of the sum \eqref{Eqn=DerivativeFinite} then implies that $f^{(n)} - \sum_{j \in \mathbb{Z}} f_j^{(n)}$ is a well-defined element of $L^\infty(\mathbb{R})$ and the series converges uniformly. Therefore, we may apply $n$ times the integral taken from $0$ to $t$ and set  
   \[
   g(t) := f(t) -  \sum_{k=0}^{n-1} \frac{t^k}{k!} f^{(k)}(0) - \sum_{j \in \mathbb{Z}} (f_j(t) -     \sum_{k=0}^{n-1} \frac{t^k}{k!} f^{(k)}_j(0) ), \qquad t \in \mathbb{R},
   \] 
   where the series converges uniformly on compact subsets of $\mathbb{R}$. 
   Then, 
   \[
   \Vert g^{(n)} \Vert_\infty \leq \Vert f^{(n)} \Vert_\infty + \sum_{j \in \mathbb{Z}} \Vert f_j^{(n)} \Vert_\infty \preceq  \Vert f^{(n)} \Vert_\infty + \Vert f  \Vert_{\Besov_{p^\sharp, p}^{n-1 + \frac{1}{p}}}.
   \]
 Now as $\sum_{j \in \mathbb{Z}}  (f_j(t) -     \sum_{k=0}^{n-1} \frac{t^k}{k!} f^{(k)}_j(0) )$ converges uniformly on compact sets and $\phi$ is a compactly supported wavelet we have that 
 \[
 \langle g, \phi_{j,k} \rangle = 0, \qquad j,k \in \mathbb{Z}. 
 \]
   The vanishing of all wavelet coefficients implies that $g$ is a polynomial $P$, see \cite[Section 6, Theorem 4 (ii)]{Bourdaud}. But as $g^{(n)}$ is uniformly bounded this polynomial $P$ must have a degree at most $n$. We conclude that,
   \[
      f(t) =  P(t)  + \sum_{j \in \mathbb{Z}} ( f_j(t)  -     \sum_{k=0}^{n-1}  \frac{t^k}{k!} f_j^{(k)}(0) ) , \qquad t \in \mathbb{R}. 
   \]
   By construction $\Vert P^{(n)} \Vert_\infty \preceq   \Vert f^{(n)} \Vert_\infty + \Vert f  \Vert_{\Besov_{p^\sharp, p}^{n-1 + \frac{1}{p}}}$.
   
\end{proof}

We now apply the results from the previous section and arrive at our main result.

\begin{proposition}\label{Prop=LocalEstimate}
 
There exists a constant $C>0$ such that for every    $f \in C^n(\mathbb{R})$ with  $\Vert f^{(n)} \Vert_\infty < \infty$ we have
\[
\Vert  f_j^{[n]}   \Vert_{\mathfrak{m}_{p_1, \ldots,  p_n}} \leq C 2^{j(n + \frac{1 }{p } - 1)}\Vert f_j \Vert_{p^\sharp}. 
\]
\end{proposition}
\begin{proof}
We apply Corollary \ref{Cor=MainWaveletEstimate} to the  decompositon \eqref{Eqn=WaveFj} and then use \cite[Lemma 4.1.2]{McDonaldSukochev} (see also \cite[Proposition 6.10.7]{MeyerBook}) to find the following inequalities that hold up to a constant independent of $j$, 
\[
\Vert  f_j^{[n]}   \Vert_{\mathfrak{m}_{p_1, \ldots,  p_n}} \preceq 2^{(n+\frac{1}{2})j} \left( \sum_{k \in \mathbb{Z}}  \vert \langle f, \phi_{j,k}  \rangle \vert^{p^\sharp} \right)^{\frac{1}{p^\sharp}} \approx 2^{(n+\frac{1}{2})j} 2^{j(\frac{1}{p^\sharp} - \frac{1}{2})}   \Vert f_j \Vert_{p^\sharp}  = 2^{j(n + \frac{1 }{p } - 1)} \Vert f_j \Vert_{p^\sharp}. 
\]
\end{proof}

 We now come to the main theorem, for which we recall all conditions to state it in a self-contained matter. 

\begin{theorem}\label{Thm=Main}
Let $n \in \mathbb{N}_{\geq 1}$, let $1 \leq p_1, \ldots, p_n < \infty$ and set $p := (p_1; \ldots ; p_n)$.
 Assume further that, 
 \begin{equation} 
 1 \leq (p_2; \ldots ; p_n), (p_1; \ldots ; p_{n-1}) < \infty.
 \end{equation}
There exists a constant $C>0$ such that for every   $f \in C^n(\mathbb{R}) \cap B_{p^\sharp, p}^{n-1 + \frac{1}{p}}$ with  $\Vert f^{(n)} \Vert_\infty < \infty$ we have,    
\[
\Vert  f^{[n]}   \Vert_{\mathfrak{m}_{p_1, \ldots,  p_n}} \leq C( \Vert f^{(n)} \Vert_\infty +  \Vert f \Vert_{B_{p^\sharp, p}^{n-1 + \frac{1}{p}}}).
\]
\end{theorem}
\begin{proof}  
By Lemma \ref{Lem=BesovApprox} there exists a polynomial $P$  of degree $\leq n$ such that
\[
f(t) = P(t) + \sum_{j \in \mathbb{Z}} (f_j(t) -      \sum_{k=0}^{n-1}  \frac{t^k}{k!} f_j^{(k)}(0)), \qquad t \in \mathbb{R},   
\]
and with $\Vert P^{(n)} \Vert_\infty \preceq \Vert f^{(n)} \Vert_\infty +  \Vert f \Vert_{B_{p^\sharp, p}^{n-1 + \frac{1}{p}}}$. 
Therefore,    
\[
f^{[n]}(t_0, \ldots, t_n) = P^{(n)}(0) +   \sum_{j \in \mathbb{Z}}  f_j^{[n]}(t_0, \ldots, t_n), \qquad t \in \mathbb{R}.
\]
Then applying the   quasi-norm estimate \eqref{Eqn=Quasi-Banach}  and \cite[Theorem 4.1.3]{McDonaldSukochev} (see \cite[Theorem 7.20]{FJW}) and Proposition \ref{Prop=LocalEstimate} gives
\[
\begin{split}
\Vert  f^{[n]}   \Vert_{\mathfrak{m}_{p_1, \ldots,  p_n}}^p \leq &  \vert P^{(n)}(0) \vert^p + \sum_{j \in \mathbb{Z}} \Vert  f^{[n]}_j   \Vert_{\mathfrak{m}_{p_1, \ldots,  p_n}}^p\\
\preceq & (\Vert f^{(n)} \Vert_\infty +  \Vert f \Vert_{B_{p^\sharp, p}^{n-1 + \frac{1}{p}}} )^p +
 \sum_{j \in \mathbb{Z}}   2^{j((n-1)p + 1)} \Vert f_j \Vert_{p^\sharp}^p \\
 \approx & (\Vert f^{(n)} \Vert_\infty +  \Vert f \Vert_{B_{p^\sharp, p}^{n-1 + \frac{1}{p}}} )^p + \Vert f \Vert_{B_{p^\sharp, p}^{n-1 + \frac{1}{p}}}^p.
 \end{split}
\]
This concludes the proof. 
\end{proof}

\section{Discussion}\label{Sect=Discussion}

We believe our main Theorem \ref{Thm=Main} gives a satisfying answer to the boundedness of multilinear operators of divided differences in case $p=1$. In the case $p \in (0, 1)$ we have obtained the first genuinely noncommutative multi-linear result where the recipient space is a quasi-Banach $L^p$-space.   What remains open is whether   our assumptions on  $p_1, \ldots, p_n$ can be relaxed upon at the expense of putting stricter regularity conditions on $\alpha$ (see Section \ref{Sect=CoreEstimate}) and therefore more regularity on our Besov space exponents. We believe such a statement should hold for general $p_1, \ldots, p_n$ coming from the interval $(0,\infty)$. 
We therefore state  Question \ref{Que=Final}. 

  We recall again that in case $n=1$ a complete answer to this problem was found by Peller \cite{Peller}; its proof is also revisited in \cite{McDonaldSukochev}.

\begin{question}\label{Que=Final}
Let $n \geq 2$. Suppose that $0< p_1, \ldots, p_n < \infty$ and let $p = (p_1; \ldots; p_n)$. 
Find parameters $a,b \in (0, \infty), s \in [1, \infty)$ such that for every $f \in \Besov_{a,b}^s$  we have 
\[
\Vert f^{[n]} \Vert_{\mathfrak{m}_{p_1, \ldots, p_n}} \preceq \Vert f^{(n)} \Vert_\infty + \Vert f \Vert_{\Besov^s_{a,b}}.
\]
\end{question}

  There are a number of difficulties emerge in attempts to address  Question  \ref{Que=Final}. The most fundamental one seems to be that  we are not able to find a multilinear analogue of the automatic complete boundedness of Schur multipliers in the quasi-Banach range as in \cite[Lemma 4.3.3 and Appendix A]{McDonaldSukochev}. This type of discretization is essential as (nonzero)   multiplication operators $\sum_{k \in \mathbb{Z}} \alpha_k \chi_{[k, k+1)} \in B(L^2(\mathbb{R}))$ do not belong to any Schatten $S_p$ class whereas their discrete analogs  $\sum_{k \in \mathbb{Z}} \alpha_k \delta_k \in  B(\ell^2(\mathbb{Z}))$ are in $S_p$  if and only if  $(\alpha_k)_{k \in \mathbb{Z}}$ is in $\ell^p$.

\end{document}